\documentclass[11pt,reqno]{amsproc}

\usepackage[margin=1in]{geometry}
\usepackage{amsmath, amsthm, amssymb}
\usepackage{graphics}
\usepackage{times}
\usepackage{wasysym}
\usepackage{hyperref}

\title{On a transport equation with nonlocal drift}

\author{Luis Silvestre}
\address{Department of Mathematics, The University of Chicago}
\email{luis@math.uchicago.edu}

\author{Vlad Vicol}
\address{Department of Mathematics, Princeton University}
\email{vvicol@math.princeton.edu}

\theoremstyle{plain}
\newtheorem{theorem}{Theorem}[section]
\newtheorem*{definition}{Definition}
\newtheorem{lemma}[theorem]{Lemma}
\newtheorem{proposition}[theorem]{Proposition}
\newtheorem{corollary}[theorem]{Corollary}
\newtheorem{conjecture}[theorem]{Conjecture}

\theoremstyle{definition}
\newtheorem{remark}[theorem]{Remark}

\numberwithin{equation}{section}

\def\RR{{\mathbb R}}

\def\ZZ{{\mathbb Z}}

\newcommand{\R}{\mathbb R}
\newcommand{\eps}{\varepsilon}
\renewcommand{\phi}{\varphi}
\renewcommand{\tilde}{\widetilde}

\newcommand{\lap} {\Delta}

\newcommand{\dx} {\; \mathrm{d} x}
\newcommand{\dd} {\; \mathrm{d}}

\DeclareMathOperator*{\osc}{osc}

\DeclareMathOperator{\sgn}{sign}

\begin{document}
\begin{abstract}
In \cite{CordobaCordobaFontelos05}, C\'ordoba, C\'ordoba, and Fontelos proved that for some initial data, the following nonlocal-drift variant of the 1D Burgers equation does not have global classical solutions
\[
\partial_t \theta +u \; \partial_x \theta = 0, \qquad  u = H \theta,
\] 
where $H$ is the Hilbert transform. We provide four essentially different proofs of this fact. Moreover, we study possible H\"older regularization effects of this equation and its consequences to the equation with diffusion
\[
\partial_t \theta + u \; \partial_x \theta + \Lambda^\gamma \theta = 0, \qquad u = H \theta,
\] 
where $\Lambda = (-\Delta)^{1/2}$, and $1/2 \leq \gamma <1$.
Our results also apply to the model with velocity field $u = \Lambda^s H \theta$, where $s \in (-1,1)$. We conjecture that solutions which arise as limits from vanishing viscosity approximations are bounded in the H\"older class in $C^{(s+1)/2}$, for all positive time. \hfill \today.
\end{abstract}

\maketitle

\section{Introduction}
\label{sec:Intro}

The question of finite time singularities for the 3D incompressible Euler equations, from smooth initial datum, is one of the fundamental problems in analysis. In the hope of understanding certain aspects of this question, motivated either by physical scenarios, numerical simulations, or simply by phenomenological analogies, over the past decades several simplified models have been proposed.  Among these, the two-dimensional  surface quasi-geostrophic (SQG) equation introduced by Constantin-Majda-Tabak~\cite{ConstantinMajdaTabak94} stands out for its striking analytic and geometric similarities to the 3D Euler equations. 
One-dimensional models have been proposed even earlier, by Constantin-Lax-Majda~\cite{ConstantinLaxMajda85}, DeGregorio~\cite{DeGregorio90,DeGregorio96}, and a number of further works~\cite{Morlet98,Sakajo03,CordobaCordobaFontelos05,ChaeCordobaCordobaFontelos05,OkamotoSakajoWunsch08,CastroCordoba08,LiRodrigo11}. See also the models very recently considered in~\cite{LuoHou13,ChoiKiselevYao13,ChoiHouKiselevLuoSverakYao14}. While for the SQG equations the question of singularities in finite time remains completely open, for the aforementioned one-dimensional models, the emergence of singularities is well understood.

In this paper we consider the 1D nonlocal transport equation introduced by C\'ordoba, C\'ordoba, and Fontelos in~\cite{CordobaCordobaFontelos05}
\begin{align}
 \partial_t \theta + u \, \partial_x \theta = 0, \qquad u = H\theta,
 \label{eq:CCF}
\end{align}
where
\begin{align}
H\theta(x) = \frac{1}{\pi} p.v. \int_{\RR} \frac{\theta(y)}{y-x} \dd y
\label{eq:Hilbert}
\end{align}
is the Hilbert transform, and $(t,x) \in [0,\infty) \times \RR$. With this convention,  $H \partial_x \theta = -\Lambda \theta$, where $\Lambda = (-\Delta)^{1/2}$. 

Phenomenologically, this equation may be viewed as a toy-model for the 2D SQG equation. However \eqref{eq:CCF} has appeared earlier in the literature in view of the strong analogies with the Birkhoff-Rott equations modeling the evolution of a vortex sheet~\cite{BakerLiMorlet96,Morlet98}. Notwithstanding the fact that this is the simplest nonlocal active scalar equation that one can write in 1D with a zero order constitutive law for the velocity, a number of open questions remain, cf.~Section~\ref{sec:Conjectures} below.

Equation \eqref{eq:CCF} lies in the middle of a scale of equations with nonlocal velocity 
\begin{align}
 \partial_t \theta + u \, \partial_x \theta  = 0, \qquad u = \Lambda^{s} H\theta,
 \label{eq:CCF:smooth:drift}
\end{align}
where $s \in (-1,1)$. This system, recently considered in~\cite{DongLi12} as a simplified model for the 2D $\alpha$-patch problem~\cite{Gancedo08}, interpolates between the classical Hamilton-Jacobi equation (for $s=1$) and a one-dimensional version of the 2D Euler vorticity equation ($s=-1$). Whereas the regularity of solutions when $s=1$ and $s=-1$ is very well understood, the intermediate cases present a number of additional difficulties due to their nonlocal nature, best exemplified by the case $s=0$.

Our interest in the model \eqref{eq:CCF} (and by extrapolation in the model \eqref{eq:CCF:smooth:drift}) also comes from an analytical point of view: the $L^\infty$ norm, which is conserved for classical solutions, may not be the strongest a-priori controlled quantity. At least for a certain class of initial data, we conjecture in Section~\ref{sec:Conjectures} that weak solutions which arise as limits of viscous regularizations have a decaying H\"older $1/2$ norm (the corresponding norm for \eqref{eq:CCF:smooth:drift} is $C^{(1+s)/2}$). While such a behavior may seem quite unintuitive given the transport nature of the equations, the phenomenon of an ``attracting regularity'' may be natural in the context of fully developed 3D turbulence. Here the Kolmogorov theory predicts that due to the forward energy cascade the $C^{1/3}$ regularity is in some sense ``stable''. We note that in the context of shell-models for 3D Euler, evidence of this phenomenon was recently obtained in~\cite{CheskidovZaya13}. Moreover, as we discuss in Section~\ref{sec:Conjectures} an a priori control on the $C^{1/2}$ norm for the inviscid problem is expected to show that the dissipative version of \eqref{eq:CCF}
\begin{align}
 \partial_t \theta + u \, \partial_x \theta + \Lambda^\gamma \theta = 0, \qquad u = H\theta,
 \label{eq:CCF:dissipative}
\end{align}
where $\gamma \geq 1/2$, has global in time smooth solutions, thereby answering Conjecture~1 in \cite{Kiselev10}. This scenario would be particularly interesting as it is not based on scaling arguments around the control obtained from the maximum principle.

Before discussing our results and the above mentioned conjectures in detail, we recall the previous works on the models \eqref{eq:CCF}, \eqref{eq:CCF:smooth:drift}, and \eqref{eq:CCF:dissipative}. The local existence of strong solutions to \eqref{eq:CCF}  was obtained in \cite{BakerLiMorlet96,Morlet98}. The emergence of finite time singularities from smooth initial datum for \eqref{eq:CCF} has been established in the remarkable work~\cite{CordobaCordobaFontelos05}. The initial datum considered there is even, non-negative and decreasing away from the origin, and the blowup scenario observed is that a cusp forms at the origin in finite time. This blowup proof was extended in~\cite{CordobaCordobaFontelos06} to cover a much wider class of initial data, via a series of new weighted integral inequalities (which are interesting in their own right). The finite time blowup is obtained from any non-constant initial data, which is natural in view of the 2-parameter scaling invariance of the equations (see Section~\ref{sec:Prelim}).  In~\cite{CastroCordoba09} the authors construct an explicit ``expanding-semicircle'' self-similar solution $\theta(t,x) = - C ( 1 - x^2/t^2)_{+}^{1/2}$ which is $C^{1/2}$ smooth, where $C>0$ is a universal constant. This is a global weak solution (in a certain sense) with almost everywhere $0$ initial datum. Numerical experiments suggest that this is a stable weak solution.

Regarding the dissipative equation~\eqref{eq:CCF:dissipative}, the global existence in the  $L^\infty$-subcritical case $\gamma>1$ was established in \cite{CordobaCordobaFontelos05}. The global well-posedness in $L^\infty$-critical case $\gamma=1$, with general initial datum was first shown in \cite{Dong08}. We note that the method developed in~\cite{Silvestre10b} for general linear drift-diffusion equations also yield this result (the methods developed in \cite{KiselevNazarovVolberg07,CaffarelliVasseur10,ConstantinVicol12,ConstantinTarfuleaVicol13} for global regularity of the critically dissipative SQG equation only appear to work in the case of positive initial datum where the $L^2$ norm of the solution is under control). The finite time blowup for the dissipative equation \eqref{eq:CCF:dissipative} with $\gamma<1/2$ was established for the first time in \cite{LiRodrigo08}, by adapting the methods developed in~\cite{CordobaCordobaFontelos05}. The inviscid and viscous blowup proofs have been revisited in~\cite{Kiselev10}, using very elegant, elementary methods, but still the question of finite time singularities in the parameter range $1/2\leq \gamma <1$ remains to date open. In comparison, for the fractal Burgers equation, i.e. $u=\theta$ in \eqref{eq:CCF}, it is known that finite time blowup occurs for any $\gamma<1$~\cite{AlibaudDroniouVovelle07,KiselevNazarovSheterenberg08,DongDuLi09,DabkowskiKiselevSilvestreVicol14}. The analogy with the Burgers equation is however tentative at best, since as opposed to $\theta\; \partial_x \theta$, the nonlinearity in \eqref{eq:CCF} is dissipative for $L^p$ norms with $1\leq p < \infty$, which hints to a regularizing mechanism.  Lastly, in~\cite{Do14} the eventual regularity for the fractionally dissipative equation is shown for the entire range $0< \gamma <1$, in the spirit of~\cite{ChanCzubakSilvestre10,Silvestre10a,Dabkowski11,Kiselev11} for the supercritically dissipative Burgers and SQG equations.

We mention that recently in~\cite{DongLi12}, motivated by analogies with the 2D $\alpha$-patch problem (see, e.g~\cite{Gancedo08}), the authors consider the system \eqref{eq:CCF:smooth:drift} and its fractionally dissipative counterpart. They prove the finite time blowup for the inviscid and slightly viscous problem and the global regularity for the critically dissipative problem. The method extends the arguments in~\cite{CordobaCordobaFontelos06,LiRodrigo08} by establishing a series of new weighted inequalities.

The main results of this paper are as follows. We give four essentially different proofs of finite time blowup from smooth initial datum for~\eqref{eq:CCF}, cf.~Theorems~\ref{p:blowup},~\ref{thm:telescope:drift},~\ref{l:globaldecay}, and~\ref{thm:barrier}. The main ideas are:
\begin{enumerate}
 \item The proof in Section~\ref{sec:CCF:revisited} is based on a new identity for $\Lambda$ (Proposition \ref{p:identity}):
\begin{align}
\label{eq:new:identity} 
\Lambda[f \Lambda f](0) + f(0) f''(0) =  \frac 12( \Lambda f(0))^2 - \frac 12 (f'(0))^2 - \left\|\frac{f(x)-f(0)}{x}\right\|_{\dot H^{1/2}}^2
\end{align}
which holds for any sufficiently smooth function $f$. This identity encapsulates the information of Theorem 1.1 (II) of \cite{CordobaCordobaFontelos06} as $\alpha \to 2$ when we apply it to $f = H\theta$ (see also Remark~\ref{rem:CCF} below). Using \eqref{eq:new:identity} we deduce that if the initial data is strictly positive at any point, then there cannot be a global $C^1$ solution. This implication is similar to how Theorem 1.1 of \cite{CordobaCordobaFontelos06} is used there to show that there is finite time blowup in \eqref{eq:CCF}. This approach is the most similar  to \cite{CordobaCordobaFontelos06} from the ones we show in this paper. It is interesting that we avoid complex integration and perhaps the identity \eqref{eq:new:identity} may also be useful in other contexts.
 \item The proof in Section~\ref{sec:Telescoping} is based on a virial proof by contradiction, in the same spirit as the proofs given in~\cite{CordobaCordobaFontelos05,CordobaCordobaFontelos06,LiRodrigo08,Kiselev10,DongLi12}. The main novelty of this proof over the aforementioned works is that we do not appeal to delicate complex analysis arguments, nor to any integration by parts in the nonlinear term. The main tool is a local in space lower bound for the nonlinearity (Lemma~\ref{lem:nonlinear:positivity}):
\begin{align}
\label{eq:new:nonlinearity:bound}
\int_{x_2}^{x_1} H\theta \, \partial_x \theta \dd x \geq \frac{1}{4\pi} \log \frac{x_1+x_2}{x_1-x_2} (\theta(x_2)-\theta(x_1))^2
\end{align}
which holds for any $0<x_2<x_1$, for $\theta$ that is even, and decreasing away from the origin (a property that is maintained by the solution of \eqref{eq:CCF} if the initial data obeys it~\cite{Kiselev10}). Estimate \eqref{eq:new:nonlinearity:bound} is then used with a suitable spacial weight for dyadic points $x_i = 2^{-i}$ in order to establish the finite time blowup of \eqref{eq:CCF}, \eqref{eq:CCF:smooth:drift}, and \eqref{eq:CCF:dissipative}.
 \item The proof in Section~\ref{sec:DeGiorgi} is based on the DeGiorgi iteration scheme. The nonlinearity dissipates the $L^1$ norm of the solution, idea which at the level of truncations $\theta_k = (\theta - C_k)^+$ yields
 \[
 \partial_t \| \theta_k \|_{L^1} + \|\theta_k\|_{\dot{H}^{1/2}}^2 \leq 0.
 \] 
Fed into the DeGiorgi iteration, the above estimate implies a decay for $\|\theta(t,\cdot)\|_{L^\infty}$ and shows that classical solutions must blowup in finite time (for classical solutions the oscillation of the solution is constant in time). 
This idea is closely related to the work of Alexis Vasseur and Chi Hin Chan on applying De Giorgi's technique on the Hamilton-Jacobi equation~\cite{VasseurChan}. In fact, we are aware that they have independently arrived to this proof for equation \eqref{eq:CCF} as well.
 \item The proof in Section~\ref{sec:Barrier} is based on constructing a suitable barrier for the solution, loosely in the spirit of a similar idea employed for the Hamilton-Jacobi equations~\cite{CardalieguetSilvestre12}. We show in Lemma~\ref{l:barrier} that  an even, monotone away from the origin, non-negative solution lies below the barrier
 \[
 \theta(t,x) \leq \theta(0,1/2) + \frac{A}{t} (1-|x|^{1/2})^+
 \]
 for some positive universal constant $A$, for all $(t,x) \in (0,\infty) \times \RR$. Again, since for classical solutions the $L^\infty$ norm is constant, it follows that classical solutions cannot live forever. This proof is perhaps the most elementary of the four blowup proofs.
\end{enumerate}
All these proofs mutatis mutandis yield the finite time blowup for~\eqref{eq:CCF:smooth:drift} in the entire parameter range $s\in (-1,1)$, recovering the results in~\cite{DongLi12}. The telescoping series proof in Section~\ref{sec:Telescoping} also directly applies to the fractionally dissipative equation~\eqref{eq:CCF:dissipative}, where it yields finite time blowup in the parameter range $\gamma < 1/2$, cf.~Theorem~\ref{thm:telescope:dissipative}, thereby recovering the results in~\cite{LiRodrigo08,Kiselev10}. The question of finite time singularities for $1/2\leq \gamma < 1$ remains open.

Lastly, in Theorem~\ref{thm:stationary} we prove that stationary solutions to the inhomogenous version of equation \eqref{eq:CCF} (respectively \eqref{eq:CCF:smooth:drift}) with bounded right hand side are $C^{1/2}$ smooth (respectively $C^{(1+s)/2}$). The proof is a direct consequence of the lower bound \eqref{eq:new:nonlinearity:bound}, and thus holds for functions $\theta$ that are even, monotone away from the origin, and non-negative. This result is motivated, and in fact also partly motivates, the discussion in Section~\ref{sec:Conjectures} in which we present two conjectures for viscosity solutions to \eqref{eq:CCF}.

We conjecture that vanishing viscosity solutions to \eqref{eq:CCF} have a bounded H\"older $1/2$ norm, for all $t>0$ (cf.~Conjecture~\ref{conj:Inviscid}). In general, we conjecture that solutions to \eqref{eq:CCF:smooth:drift} are controlled in $C^{(1-s)/2}$. In support of this conjecture, we note that none of the different blowup proofs in this paper necessarily imply the blowup of any $C^\alpha$ norm of the solution to \eqref{eq:CCF} with $\alpha\leq 1/2$. In fact, there are two explicit solutions of \eqref{eq:CCF} that have a singularity of order exactly H\"older $1/2$: 
\[
\theta(t,x) = - |x|^{1/2} - C_1 t \qquad  \mbox{and}  \qquad \theta(t,x) = - C_2 (1 - x^2/t^2)_{+}^{1/2},
\] 
where the $C_i>0$ are suitable universal constants. It is not clear however whether these solutions may be obtained as limits of vanishing viscosity approximations to \eqref{eq:CCF}. Besides the result in Theorem~\ref{thm:stationary} which shows that certain stationary solutions do in fact obey $C^{1/2}$ bounds, a last argument in favor of such a regularizing effect, perhaps similar to the one for the Hamilton-Jacobi equations, is suggested by the numerical experiments which rely on the code available at \url{http://math.uchicago.edu/~luis/pde/hilbert.html}. In particular, since we expect this regularizing phenomenon to be still valid for the fractionally dissipative equations, it would imply that the equation \eqref{eq:CCF:dissipative} with $\gamma \geq 1/2$ has global smooth solutions (cf.~\cite{Silvestre10b}), thereby answering Conjecture~1 in \cite{Kiselev10}. 

We conclude the paper with a related and perhaps weaker conjecture: that for vanishing viscosity solutions of \eqref{eq:CCF} we have a lower bound $\Lambda \theta(t,x) \geq - A(t,\theta_0)$, for all time $t>0$ and some decaying function $A(t,\theta_0)>0$ (cf.~Conjecture~\ref{conj:Lambda:lower:bound}). Geometrically, such a lower bound holds if the cusps that form in the solution always point up, phenomenon supported by the numerical simulations. In fact we show in Theorem~\ref{thm:Holder} that if this conjecture holds, when $\theta$ is even, non-negative, and decaying away from the origin, then $\theta$ must be H\"older $1/2$ smooth away from the origin. We believe that a lower bound for $\Lambda \theta$ is in fact related to the uniqueness of weak solutions to \eqref{eq:CCF}, in the same way a one-sided bound for the derivative works for the Burgers equation and the one sided-bound on the second-derivative works for the Hamilton-Jacobi equations~\cite{Evans98}.

\section{Preliminaries}
\label{sec:Prelim}
In this section we recall some preliminary observations regarding the initial value problem for \eqref{eq:CCF}.

\subsection{Scaling}
If $\theta$ is a solution to \eqref{eq:CCF}, then for any values of $a,b \in \R$ the rescaled function 
\[ 
\theta_{ab}(t,x) = ab^{-1} \theta(at,bx)
\]
is also a solution to \eqref{eq:CCF}.

For the equation with general drift~\eqref{eq:CCF:smooth:drift} the rescaling is given by $\theta_{ab}(t,x) = ab^{-1-s} \theta(at,bx)$, while for the dissipative equation \eqref{eq:CCF:smooth:drift} the rescaling is given by
$\theta_b(t,x) = b^{\gamma-1}\theta(b^\gamma t, bx)$. Note that in the later case we have a one-parameter family. 

\subsection{Maximum principles}

We first state the $L^\infty$ maximum principle, which follows directly from the transport structure of \eqref{eq:CCF}.

\begin{lemma} \label{lem:maxprinciple}
Let $\theta$ solve \eqref{eq:CCF}. Then $\sup_\R \theta$ is non increasing in time and $\inf_\R \theta$ is non decreasing in time. In particular $\osc_\R \theta$ is non increasing.
\end{lemma}
The maximum principle stated in Lemma~\ref{lem:maxprinciple} also holds for solutions to \eqref{eq:CCF:smooth:drift} and \eqref{eq:CCF:dissipative}. In the first case because we have a transport equation, and in the second case because the fractional Laplacian with $\gamma \in (0,2]$ has a maxmimum/minimum principle.

There is no obvious notion of a weak solution to \eqref{eq:CCF}. We assume that there is a tiny extra viscosity term $\eps \lap \theta$ on the right side of \eqref{eq:CCF}, so that the solution is classical, and obtain bounds independent of $\eps$. This vanishing viscosity makes the equation time-irreversible. In particular, the $L^\infty$ norm of $\theta$ will be non-increasing in time instead of constant.

There is a mildly stronger version of the above maximum principle, stated below for completeness.
\begin{lemma}\label{lem:maxprincipleW11}
Let $\theta$ be a classical, i.e. $C^{1,\alpha}$ smooth, decaying at infinity, solution of \eqref{eq:CCF}. The quantity
\[ \int_{\R} |\theta_x(t,x)| \dd x \]
is constant in time.
\end{lemma}

\begin{proof}
Consider $\{ \phi_\eps(\cdot)\}_{\eps>0}$ a family of non-negative smooth convex functions, with $\phi_\eps(0)=0$, which converge to the absolute value function $|\cdot|$, as $\eps \to 0$. Multiplying the equation obeyed by $\theta_x$ with $\phi_\eps'(\theta_x)$ and integrating over space we arrive at 
\[
\partial_t \int_{\RR} \phi_\eps(\theta_x) \dd x + \int_{\RR} H \theta_x \left( \theta_x \phi_\eps'(\theta_x) - \phi_\eps(\theta_x) \right) \dd x = 0.
\]
The boundary terms arising in the above computation vanish since we assume $\theta_x \to 0$ as $|x|\to \infty$, and we considered $\phi_\eps(0)=0$. The proof is complete by passing $\eps \to 0$ since $\phi_\eps(x) - x \phi_\eps'(x) \to 0$ pointwise. 
\end{proof}
Using the convexity assumption on $\phi_\eps$ one can show that for solutions which arise as limits of viscous regularizations, the $W^{1,1}$ nor is non-increasing in time.

\subsection{Solutions with symmetries}
In Sections~\ref{sec:Telescoping}, \ref{sec:Barrier}, and parts of Section~\ref{sec:Conjectures} we restrict the attention to initial data $\theta_0$ that is smooth, non-negative, decaying at infinity, even in $x$, and with $\sgn(x)\, \partial_x \theta_0 \leq 0$. The (viscosity) solution will inherit all these properties for positive time (see Lemma~\ref{lem:fractional:good:properties}), except possibly the smoothness. If $\theta$ is an even function, the formula \eqref{eq:Hilbert} becomes
\begin{align}
H\theta (x) = \frac{2x}{\pi} \int_0^\infty \frac{\theta(y)-\theta(x)}{y^2-x^2} \dd y.
\label{eq:Hilbert:even}
\end{align}
Note that if $\theta$ is even and monotone decreasing away from the origin, then $H\theta(x) < 0$ {\em at every point} $x>0$ unless $\theta$ is constant.

\section{Singularity formation via identities for the Zygmund operator}
\label{sec:CCF:revisited}

In this section we provide our first proof that some solutions to the equation \eqref{eq:CCF} must develop a singularity in finite time. 
Before turning to the proof of Theorem~\ref{p:blowup} we use a number of identities for the Hilbert transform in order to obtain a new identity for $\Lambda$, cf.~Proposition~\ref{p:identity} below.

\subsection{Preliminary identities around the Hilbert transform}
\begin{lemma} \label{l:i1}
For any function $g: \R \to \R$,
\[ H[xg(x)] = x Hg(x) + \int_{\R} g(x) \dd x,\]
\end{lemma}

\begin{proof}
We compute $H[xg(x)]$ directly.
\begin{align*}
H[xg(x)] &= p.v. \int_{\R} \frac{ y g(y) }{y-x} \dd y, \\
&= x Hg(x) + p.v. \int_{\R} \frac{ (y-x) g(y) }{y-x} \dd y, \\
&= x Hg(x) + p.v. \int_{\R}  g(y) \dd y. 
\end{align*}
\end{proof}

\begin{remark}
Note that the function $Hg(x) + i g(x)$ is the value on the real line of a holomorphic function $f(z)$ in the upper half space. The function $x Hg(x) + i x g(x)$ is the value on the real line of $z f(z)$. This suggests the formula $H[xg(x)] = x Hg(x)$, which is correct up to the addition of a constant term.

The Hilbert transform is a well defined isometric isomorphism in $L^2(\R)$. Also $Hf$ makes sense for many functions $f$ which are not in $L^2$. If may extend it to any function $f$ which is the imaginary part of the value on $\R$ of a holomorphic function in the upper half space. 
\end{remark}

\begin{corollary} \label{c:i1}
For any function $g: \R \to \R$,
\[ \Lambda[x g(x)] = -Hg(x) + x \Lambda g(x).\]
\end{corollary}

\begin{proof}
The corollary follows by differentiating the identity in Lemma \ref{l:i1}.
\end{proof}

\begin{lemma} \label{l:i2}
For any function $g: \R \to \R$,
\[ H[g Hg] = \frac 12 Hg^2 - \frac 12 g^2.\]
\end{lemma}

\begin{proof}
It is a direct consequence of $(Hg+ig)^2 = (Hg^2 - g^2) + 2i g Hg$.
\end{proof}

\subsection{An identity for \texorpdfstring{$\Lambda$}{Lambda}}

\begin{proposition} \label{p:identity}
Let $f:\R \to \R$ be any function. Then
\[ \Lambda[f \Lambda f](0) + f(0) f''(0) =  \frac 12( \Lambda f(0))^2 - \frac 12 (f'(0))^2 - \left\|\frac{f(x)-f(0)}{x}\right\|_{\dot H^{1/2}}^2.\]
\end{proposition}

\begin{proof}
By subtracting a constant, we can assume without loss of generality that $f(0)=0$, in which case the second term on the left hand  term vanishes.

Assuming $f(0) = 0$, we want to compute
\begin{equation} \label{e:i1}  \Lambda[f \Lambda f](0) = - p.v. \int_\R \frac{f(x) \Lambda f(x)}{x^2} \dd x.
\end{equation}
Let $f(x) = x g(x)$. Since $f(0)=0$, we can always find such function $g$. Moreover, we observe that
\[ g(0) = f'(0) \text{ and } Hg(0) = -\Lambda f(0).\]
We rewrite \eqref{e:i1} in terms of $g$ using Corollary \ref{c:i1}.
\begin{align*}
\Lambda[f \Lambda f](0) &= - p.v. \int_{\R} \frac{-x g(x) Hg(x) + x^2 g(x) \Lambda g(x)}{x^2} \dd x,\\
&= p.v. \int_{\R} \frac{g(x) Hg(x)}{x} \dd x - p.v. \int_{\R} g(x) \Lambda g(x) \dd x, \\
&= H[gHg](0) - \|g\|_{\dot H^{1/2}}^2.
\end{align*}
We use Lemma \ref{l:i2} to rewrite the term $H[gHg]$ and arrive at
\begin{align*}
\Lambda[f \Lambda f](0) &= \frac 12 (Hg(0))^2 - \frac 12 (g(0))^2 - \|g\|_{\dot H^{1/2}}^2.
\end{align*}
Recalling that $g(0) = f'(0)$ and $Hg(0) = -\Lambda f(0)$, we obtain the identity
\[ 
\Lambda[f \Lambda f](0) = \frac 12 (\Lambda f(0))^2 - \frac 12 (f'(0))^2 - \|f(x)/x\|_{\dot H^{1/2}}^2
\]
which concludes the proof.
\end{proof}

\begin{remark}
\label{rem:CCF}
The right hand side in Proposition \ref{p:identity} does not have any particular sign. Indeed, if $g$ is the function as in the proof of Proposition \ref{p:identity}, we get
\begin{align*}
\Lambda[f \Lambda f](0) &= \frac 12 (Hg(0))^2 - \frac 12 (g(0))^2 - \|g\|_{\dot H^{1/2}}^2, \\
&= \frac 12 (Hg(0))^2 - \frac 12 (g(0))^2 - \|Hg\|_{\dot H^{1/2}}^2.
\end{align*}
Choosing $g$ odd, we would have $g(0)=0$. Thus $Hg$ would be an arbitrary even function, and in general the $\|Hg\|_{\dot H^{1/2}}$ norm canot control the value of $Hg(0)$.

Note that this does not contradict the result in \cite{CordobaCordobaFontelos06} since there is a typo in Theorem 1.1 (II). The assumption in that theorem is meant to say that $f- f(0)$ is a nonnegative (or nonpositive) function.
\end{remark}

\subsection{Blow-up proof using Hilbert transform identity}

\begin{theorem} \label{p:blowup}
Let $\theta_0$ be any function which converges to zero as $x \to \pm \infty$ and has a positive global maximum at a point $x_0$. Then, a $C^1$ solution to \eqref{eq:CCF} such that $\theta(0,x) = \theta_0$, cannot exist for all time.
\end{theorem}

\begin{proof}
For $x_0$ be the point so that 
\[ \theta(0,x_0) = \max_{x \in \R} \theta(0,x).\]
Let us follow the flow
\begin{align*}
\dot X(t) &= H\theta(t,X(t)),\\
X(0) &= x_0.
\end{align*}
The value of $\theta$ is constant along the flow of the transport equation, therefore
\[ \theta(t,X(t)) = \max_{x \in \R} \theta(t,x).\]
In particular $\theta_x(t,X(t))=0$ and $\Lambda \theta(t,X(t))>0$.

Now we compute the evolution of $\Lambda \theta(t,X(t))$. We have
\begin{align*}
\partial_t \Lambda \theta(t,X(t)) &= \Lambda \theta_t + H\theta \ \Lambda \theta_x = -\Lambda [H\theta \theta_x] + H\theta \ \Lambda \theta_x.
\end{align*}
Note that $H\theta \theta_x = f \Lambda f$ and $\Lambda \theta_x = - f_{xx}$ for $f = H \theta$. We apply Proposition \ref{p:identity} and obtain.
\begin{align*}
\partial_t \Lambda \theta(t,X(t)) &= -\frac 12 (\theta_x)^2 + \frac 12 (\Lambda \theta)^2 + \| H\theta/x \|_{\dot H^{1/2}}^2 \geq \frac 12 (\Lambda \theta(t,X(t)))^2
\end{align*}
since $\theta_x(t,X(t))=0$.
This ODE for $\Lambda \theta(t,X(t))$ blows up in finite time, thereby concluding the proof.
\end{proof}

\begin{remark}
The previous proof can also be applied at any initial point $x_0$ which is a local max or min for which $\Lambda \theta(x_0,0)>0$. In particular it is possible to find an even initial condition $\theta_0$ which develops a singularity away from the origin.
\end{remark}

\begin{remark}
\label{rem:Lambda:lower:bound}
Note that the computation in the previous proof shows that for smooth solutions of \eqref{eq:CCF}, $\Lambda \theta$  obeys the PDE
\[ \partial_t \Lambda \theta + H\theta (\Lambda \theta)_x = -\frac 12 (\theta_x)^2 + \frac 12 (\Lambda \theta)^2 + \left\| \frac{ H\theta(\cdot)-H(x) }{\cdot-x} \right\|_{\dot H^{1/2}}^2. \]
The second and third term in the right hand side are positive, but the first one is negative. If the right hand side was non negative, it would imply that a lower bound on $\Lambda \theta$ is preserved by the flow.

It is currently not clear whether $\Lambda \theta$ remains bounded from below for positive time. This issue is addressed in Section~\ref{sec:Conjectures} below, in connection with a conjectured a priori estimate for $\theta$ in $C^{1/2}$ (see Conjecture~\ref{conj:Lambda:lower:bound}).
\end{remark}

\section{Singularity formation via telescoping sums}
\label{sec:Telescoping}

In this section we present our second proof that smooth solutions to \eqref{eq:CCF} cannot exist for all time  (Theorem~\ref{thm:telescope} below). The proof is based on a local in space lower bound for the nonlinearity, which is established in Lemma~\ref{lem:nonlinear:positivity}.

\subsection{An integral bound for the nonlinearity}
\begin{lemma}
\label{lem:pointwise:Hilbert}
Assume $\theta$ is even and decreasing away from the origin. Then we have
\begin{align}
 - H \theta(x_2) \geq \frac{1}{\pi} \log\left( \frac{x_1+x_2}{x_1 - x_2} \right) (\theta(x_2) - \theta(x_1) ) \geq 0
 \label{eq:pointwise:Hilbert}
\end{align}
for any $0<x_2<x_1$.
\end{lemma}
\begin{proof}
Using \eqref{eq:Hilbert:even} and monotonicity we have the bound
\begin{align*}
- H\theta(x_2) &= \frac{2x_2}{\pi} \int_0^\infty \frac{\theta(x_2)-\theta(y)}{y^2-x_2^2} \dd y \geq \frac{2x_2}{\pi} \int_{x_1}^\infty \frac{\theta(x_2)-\theta(y)}{y^2-x_2^2} \dd y \notag\\
&\geq \frac{2x_2}{\pi} (\theta(x_2)-\theta(x_1) ) \int_{x_1}^\infty \frac{\dd y}{y^2-x_2^2}  = \frac{1}{\pi} (\theta(x_2)-\theta(x_1) ) \log\left( \frac{x_1+x_2}{x_1 - x_2} \right)
\end{align*}
and the proof is complete.
\end{proof}
The above estimate was previously used in~\cite{Kiselev10} in order to give a different proof of the inequality in~\cite{CordobaCordobaFontelos06}, which avoids the use of subtle complex analysis arguments.
The following Lemma yields a lower bound  for the nonlinear term in \eqref{eq:CCF}, which is local in nature.
\begin{lemma}
\label{lem:nonlinear:positivity}
Let $\theta$ be smooth, even, and decreasing away from the origin. Then we have
\begin{align}
\int_{x_2}^{x_1} H\theta \, \partial_x \theta \dd x \geq \frac{1}{4\pi} \log \frac{x_1+x_2}{x_1-x_2} (\theta(x_2)-\theta(x_1))^2
\label{eq:nonlinear:positivity}
\end{align}
for any $0<x_2<x_1$.
\end{lemma}
\begin{proof}
Note that $H\theta \, \partial_x \theta \geq 0$ by our assumptions on $\theta$.
Choose a point $x_3 \in (x_2,x_1)$ such that $2 \theta(x_3) = \theta(x_1) + \theta(x_2)$. Without loss of generality we have $2 (x_3 - x_2) \leq x_1 - x_2$, as the other case can be treated similarly. 

Using the monotonicity of $\theta$ and applying the lower bound \eqref{eq:pointwise:Hilbert} to a point $x \in (x_3,x_1)$, we obtain
\[
- H \theta(x) 
\geq \frac{1}{\pi} ( \theta(x) - \theta(x_1) ) \log\left( \frac{x_1+x}{x_1- x} \right) 
\geq \frac{1}{2\pi}  (\theta(x_2)-\theta(x_1))\log \left( \frac{x_1+x_2}{x_1-x_2} \right).
\]
Now, since $-\theta_x \geq 0$, we obtain
\begin{align*}
\int_{x_2}^{x_1} H\theta \, \partial_x \theta \dd x 
\geq \int_{x_2}^{x_3} H\theta \, \partial_x \theta \dd x 
&\geq \frac{1}{2\pi}  (\theta(x_2)-\theta(x_1))\log \left( \frac{x_1+x_2}{x_1-x_2} \right) \int_{x_2}^{x_3} ( - \partial_x \theta) \dd x\notag\\
&= \frac{1}{4\pi}  (\theta(x_2)-\theta(x_1))^2 \log \left( \frac{x_1+x_2}{x_1-x_2} \right) 
\end{align*}
where we have also used the definition of $x_3$.
Note that in the proof we did not integrate by parts.
\end{proof}

In particular, the above Lemma is a local version of the identity
\[
\int_{\RR} \theta_x H\theta \dd x = \|\theta\|_{\dot{H}^{1/2}}^2 = c \int \! \! \! \int_{\RR} \frac{(\theta(x_1)-\theta(x_2))^2}{(x_1-x_2)^2} \dd x_1 \dd x_2\geq 0
\]
which follows from integration by parts.

\subsection{The inviscid case}

\begin{theorem}
\label{thm:telescope}
Let $\theta_0$ be even, non-negative, monotone decreasing on $(0,\infty)$. Then the initial value problem for  \eqref{eq:CCF} does not have a global in time $C^{1}_x$ smooth solution.
\end{theorem}
The idea of the proof  to use a weighted version of Lemma~\ref{lem:nonlinear:positivity} in a dyadic fashion. 
The proof works for {any} smooth initial datum $\theta_0$ that decays at infinity, and has $\theta_0(0)>0$.

\begin{proof}
Assume the ensuing solution $\theta$ of \eqref{eq:CCF} lies in $L^\infty(0,T;C^{1})$ for some $T>0$. Then $\theta(\cdot,t)$ is even, decreasing away from the origin and non-negative on $[0,T)$. The following computations are then justified on this time interval. We will arrive at a contradiction if $T$ is sufficiently large, which implies that the $C^{1}$ norm of the solution must blow up in finite time.

Consider the continuous function $\eta \colon (0,\infty) \to (0,\infty)$ defined by 
\begin{align}
\eta(x) = 
\begin{cases}
x^{-\alpha}&, 0 < x < 1, \\
x^{-2-\alpha}&, x>1,
\end{cases}
\label{eq:eta:def}
\end{align}
for some $\alpha \in (0,1)$. Then $\eta \in L^1(\RR)$, is even and monotonically decreases away from the origin. Define the Lyapunov functional 
\begin{align}
F(t) = F_{\theta_0}(t) 
= \int_0^\infty \eta(x) (\theta(0,t) -  \theta(x,t) ) \dd x \geq 0.
\label{eq:F:def}
\end{align}
Then, in view of the $L^\infty$ maximum principle of Lemma~\ref{lem:maxprinciple}, we have that
\begin{align}
F_{\theta_0}(t) \leq \|\theta_0\|_{L^\infty} \|\eta\|_{L^1} = \frac{\|\theta_0\|_{L^\infty}}{\alpha(1-\alpha)}.
\label{eq:F:apriori}
\end{align}
We will now use the equation \eqref{eq:CCF} to deduce that $F$ obeys an ODE which blows up in finite time, if $F(0) > 0$.
Differentiating in time, and using that as long $\theta$ remains H\"older continuous we must have $H \theta(0,t) =  0$, we obtain
\begin{align}
\frac{dF}{dt}(t) &= \int_0^\infty \eta(x) \partial_t (\theta(t,0) - \theta(t,x)) \dd x  \notag\\
&=  \int_0^\infty \eta(x) H\theta(t,x) \, \partial_x \theta(t,x) \dd x  - H\theta(t,0) \partial_x \theta(t,0) \int_0^\infty\eta(x) \dd x\notag\\
&=  \sum_{k \in \ZZ} \int_{2^{k}}^{2^{k+1}} \eta(x) H\theta(t,x) \, \partial_x \theta(t,x) \dd x.
\label{eq:F:dot:1}
\end{align}
We now appeal to Lemma~\ref{lem:nonlinear:positivity}  and obtain
\[
\int_{2^{k}}^{2^{k+1}} H\theta \, \partial_x \theta \dd x \geq \frac{\log 3}{4\pi} (\theta(2^{k}) - \theta(2^{k+1}) )^2.
\]
Combining the above estimate with \eqref{eq:F:dot:1} we obtain
\begin{align}
\frac{dF}{dt} (t) \geq  \frac{\log 3}{4\pi} \sum_{k\in \ZZ} \eta(2^{k+1})  (\theta(t,2^{k}) - \theta(t,2^{k+1}) )^2
\label{eq:dtF:lower}
\end{align}
where we have also used that $\eta$ is decreasing.

On the other hand we may write 
\[
\eta(x) = - \partial_x  \phi(x)
\]
for all $x\in (0,1) \cup (1,\infty)$ where
\[ \phi(x) = 
\begin{cases}
\frac{1}{1+\alpha} + \frac{1}{1-\alpha} (1- x^{1-\alpha}) &, 0 < x < 1, \\
\frac{1}{1+\alpha} x^{-1-\alpha}&, x>1.
\end{cases}
\] 
Note that $\phi \geq 0$ is monotone decreasing.
Therefore, we have
\begin{align}
F(t) &= - \int_0^\infty \partial_x \phi(x) (\theta(t,0) - \theta(t,x)) \dd x =  - \int_0^\infty \phi(x) \partial_x \theta(x,t) \dd x \notag \\
&=  \sum_{k \in \ZZ} \int_{2^k}^{2^{k+1}} \phi(x) (- \partial_x \theta(t,x)) \dd x  \leq \sum_{k \in \ZZ} \phi(2^{k}) (\theta(t,2^k)  - \theta(t,2^{k+1}))
\label{eq:F:upper:ancient}
\end{align}
where in the second to last line we have used that $\theta$ and $\phi$ are decreasing. At last, using the Cauchy-Schwartz inequality, it follows from \eqref{eq:F:upper:ancient} that 
\begin{align}
F(t) &\leq \sum_{k \in \ZZ} \phi(2^{k}) (\eta(2^{k+1}))^{-1/2} (\theta(t,2^k)  - \theta(t,2^{k+1}))  (\eta(2^{k+1}))^{1/2} \notag\\
&\leq \left( \sum_{k\in \ZZ} (\phi(2^{k}))^2 (\eta(2^{k+1}))^{-1} \right)^{1/2} \left( \sum_{k \in \ZZ} \eta(2^{k+1})  (\theta(t,2^k)  - \theta(t,2^{k+1}))^2 \right)^{1/2}.
\label{eq:F:upper}
\end{align}
In view of the choice in of $\eta$ (and thus $\phi$), we have that
\[
\sum_{k\in \ZZ} \frac{(\phi(2^{k}))^2}{ \eta(2^{k+1})} \leq \frac{2^{1+\alpha}}{1-\alpha^2} \sum_{k < 0} 2^{\alpha k} + \frac{2^{2+\alpha}}{(1+\alpha)^2} \sum_{k \geq 0} 2^{- \alpha k}  = c_\alpha
\]
where $c_\alpha>0$ is an explicitly computable constant. 

Therefore, by combining \eqref{eq:dtF:lower} with \eqref{eq:F:upper} we obtain that 
\begin{align}
 \frac{dF}{dt}(t) \geq \frac{\log 3}{4 \pi c_\alpha} (F(t))^2 
\label{eq:F:dot:2}
\end{align}
which implies a finite time blowup for $F$, since $F(0) > 0$, for {\em any} $\theta_0$ which is not a constant. This contradicts \eqref{eq:F:apriori} and thus completes the proof of the theorem.
\end{proof}

\subsection{The fractionally transport velocity case}
\begin{theorem}
\label{thm:telescope:drift}
Let $\theta_0$ be even, non-negative, monotone decreasing on $(0,\infty)$. Let $s \in (-1,1)$. Then the initial value problem for \eqref{eq:CCF:smooth:drift} cannot have a global in time $C^{1}_x$ smooth solution.
\end{theorem}

Before giving the proof of the above statement we note that estimate \eqref{eq:nonlinear:positivity} has a direct analogue for the case of a general drift velocity $H \Lambda^s \theta$. 
\begin{lemma}\label{lem:nonlinear:positivity:drift}
 Under the assumptions of Lemma~\ref{lem:nonlinear:positivity}, when $s\in (-1,1)$ we have
 \begin{align}
\int_{x_2}^{x_1} H \Lambda^s \theta \, \partial_x \theta \dd x \geq \frac{1-s}{16s} \left( 1 - \frac{(x_1-x_2)^s}{(x_1+x_2)^s} \right) \frac{(\theta(x_2)-\theta(x_1))^2}{(x_1-x_2)^s}
\label{eq:nonlinear:positivity:drift}
\end{align}
for any $0<x_2<x_1$.
\end{lemma}
\begin{proof}
 The proof is the same as the one of Lemma~\ref{lem:nonlinear:positivity}, except that instead of the pointwise estimate \eqref{eq:pointwise:Hilbert}, we use the bound
 \begin{align}
 - H \Lambda^s \theta(x_2) \geq \frac{(1-s)}{4 s} \left( 1 - \frac{(x_1-x_2)^s}{(x_1+x_2)^s} \right) \frac{\theta(x_2) - \theta(x_1)}{(x_1-x_2)^s}\geq 0
 \label{eq:pointwise:Hilbert:s}
\end{align}
which holds for $0 \neq s\in (-1,1)$ and any $0<x_2<x_1$.
Note that in the limit $s\to 0$, be bound \eqref{eq:pointwise:Hilbert:s} is consistent with estimate \eqref{eq:pointwise:Hilbert}.

In order to prove \eqref{eq:pointwise:Hilbert:s}, recall that $H \Lambda^s \theta = \partial_x \Lambda^{s-1} \theta$. Since $s-1<0$, the operator $\Lambda^{s-1}$ is given by convolution with the Riesz potential
\[
\Lambda^{s-1} \theta(x) = \frac{1}{c_s} \int_{\RR} \frac{\theta(y)}{|y-x|^{s}} \dd y
\]
where $c_s = \sqrt{\pi} 2^{1-s} \Gamma((1-s)/2) / \Gamma(s/2)$.
Therefore we have
\[
H \Lambda^s \theta(x) = \frac{s}{c_s} P.V. \int_{\RR} \frac{\theta(y)}{(y-x)|y-x|^s} \dd y = \frac{s}{(1-s)c_s} (1-s) \int_{\RR} \frac{\theta(y)-\theta(x)}{(y-x)|y-x|^s} \dd y.
\]
One may verify explicitly that 
\[
\frac{1}{4} \leq \frac{s}{(1-s)c_s} \leq \frac 12
\]
for $s\in(-1,1)$, and that the dependence on $s$ is monotone increasing.

The proof now follows just as the proof of Lemma~\ref{lem:pointwise:Hilbert}. Under the standing assumptions on $\theta$ we have
\begin{align*}
 - H \Lambda^s \theta(x_2) 
 &\geq \frac{1-s}{4} (\theta(x_2) - \theta(x_1)) \int_{x_1}^{\infty} \frac{1}{(y-x_2)^{s+1}} - \frac{1}{(y+x_2)^{s+1}} \dd y.
\end{align*}
The integral on the right side of the above converges for $s\in (-1,1)$ and after a direct computation we obtain
\[
- H \Lambda^s \theta(x_2)  \geq \frac{1-s}{4s} (\theta(x_2) - \theta(x_1)) \left( \frac{1}{(x_1-x_2)^{s}} - \frac{1}{(x_1+x_2)^s} \right)
\]
which concludes the proof.
\end{proof}

\begin{proof}[Proof of Theorem~\ref{thm:telescope:drift}]
The proof is a slight modification of the proof of Theorem~\ref{thm:telescope}.
Note that the monotonicity and symmetry properties of the solution are maintained by the flow \eqref{eq:CCF:smooth:drift} in view of its transport nature. Consider the function $\eta(x)$ defined in \eqref{eq:eta:def}, and let $F(t)$ be defined by \eqref{eq:F:def}.
In view of Lemma~\ref{lem:nonlinear:positivity:drift}, we have that 
\[
\int_{2^k}^{2^{k+1}} H\Lambda^s \theta \partial_x \theta \dd x 
\geq \frac{1-s}{16s} 
\left( 1 - 3^{-s} \right) 2^{-ks} (\theta(2^k)-\theta(2^{k+1}))^2,
\]
so that similarly to \eqref{eq:dtF:lower}, we have
\[
\frac{dF}{dt}(t) \geq c_s \sum_{k \in \ZZ} \eta(2^{k+1}) 2^{-ks} (\theta(2^k)-\theta(2^{k+1}))^2
\]
for some constant $c_s >0$ that depends only on $s\in (-1,1)$.
In view of the bound \eqref{eq:F:upper:ancient}, the proof of the Theorem is completed once we establish that
\begin{align}
\sum_{k\in \ZZ} \frac{( \phi(2^k) )^2}{\eta(2^{k+1}) 2^{-ks}} < \infty
\label{eq:telescope:s:to:check}
\end{align}
for some $\alpha \in (0,1)$. In view of the definitions of $\phi$ and $\eta$, the estimate \eqref{eq:telescope:s:to:check} follows from
\[
\sum_{k<0} 2^{(\alpha+s) k} + \sum_{k\geq 0} 2^{(-\alpha+s)k} < \infty
\]
which holds once we let 
\[
\alpha > |s|
\]
which is consistent with $\alpha \in (0,1)$.
\end{proof}

\subsection{The fractionally dissipative case}

We now show how the telescopic sum argument can also be used to prove the emergence of singularities in finite time for a problem with fractional diffusion. The proof is in the same spirit as the proofs given in~\cite{LiRodrigo08} and \cite{Kiselev10,DongLi12}, but as in the inviscid case we avoid using the integral inequality of~\cite{CordobaCordobaFontelos06}.

The symmetry and monotonicity of the function $\theta$ plays an important role in this proof. It is easy to see that if the initial value $\theta_0$ is even and non negative, the solution $\theta(\cdot,t)$ will stay even and non negative for all values of $t>0$. It is also true that if $\theta_0$ is monotone decreasing away from the origin, the same property holds for $\theta(\cdot,t)$ for all $t>0$. The preservation of this property is perhaps the least obvious one given the influence of the fractional diffusion. {This fact was already presented in \cite[Lemma 6.3]{Kiselev10}. Since the proof in~\cite{Kiselev10}  discusses only the case where the sign of $\theta_x$ is lost away from the origin, for the sake of completeness we give here the proof of this fact, which we state in the following lemma.}

\begin{lemma}
\label{lem:fractional:good:properties}
Assume $\theta_0$ is even, smooth, and non increasing away from the origin. Let $\theta$ be a smooth, decaying at infinity
solution to the problem \eqref{eq:CCF:dissipative}. Then $\theta(\cdot,t)$ is even and non increasing away from the origin for any $t>0$.
\end{lemma}

\begin{proof}
The evenness of $\theta$ is conserved due to invariance under $x\mapsto -x$ of the equation. The more delicate part is to prove that $\partial_x \theta(t,x) \leq 0$ for all $t \in [0,T)$ and all $x\geq 0$. Note that in view of the 
evenness of $\theta(t,\cdot)$, we have that $\partial_x \theta(t,0) =0$. 

Assume by contradiction that for some $t_0 \in [0,T)$ and some $x_0 \in (0,\infty)$ we have that $\partial_x\theta(t_0,x_0) > 0$. Fix $A = 2 \sup_{t \leq t_0} \|\Lambda \theta(t,\cdot)\|_{L^\infty}$ and pick $\eps>0$ sufficiently small so that $\partial_x\theta(t_0,x_0) > \eps e^{At_0}$.

The derivative $\theta_x = \partial_x \theta$ satisfies the equation
\begin{align}
\partial_t (\theta_x) + H\theta \, \partial_x (\theta_x) - \Lambda \theta \, (\theta_x) + \Lambda^\gamma (\theta_x) = 0.
\label{eq:theta:x:evolution}
\end{align}
Let $(t_1,x_1) \in (0,\infty)\times(0,\infty)$ be the first crossing point between the functions $\theta_x(t,x)$ and $\eps e^{At}\sgn(x)$. That means that $\theta_x(t,x) < \eps e^{At}$ for all $x \in [0,\infty)$ and $t < t_1$, but that $\theta_x(t_1,x_1) = \eps e^{At_1}$. This point must exist because we know that $\theta_x(0,x) < \eps$ for $x\geq 0$, and by assumption $\theta_x(t_0,x_0) > \eps e^{At_0}$. Moreover, $\theta_x(t,0)=0$ for all $t$ and $\theta_x$ is continuous, so that $\theta_x(t,x)< \eps e^{At}$ when $x$ is too close to the origin. This ensures that $x_1>0$. Finally, we assume $\theta_x(t,x) \to 0$ as $x \to \infty$ (this decay assumption is not strictly necessary, but it makes the proof easier).

Now we evaluate the equation for $\theta_x$ at the point $(t_1,x_1)$ and obtain a contradiction. By the minimality of $t_1$ we must have $\partial_t(\theta_x)(t_1,x_1) \geq \partial_t(\eps e^{At})_{|{t=t_1}} = A\eps e^{At_1}$. Note that $\theta_x(t_1,\cdot)$ achieves its maximum on $[0,\infty)$ at $x=x_1$, which yields that $\partial_x(\theta_x)(t_1,x_1) = 0$. Moreover, when we compute $\Lambda^\gamma(\theta_x(t_1,x_1))$. Since $\theta_x$ is odd we have
\begin{align*}
\Lambda^\gamma \theta_x(t_1,x_1) 
&= c P.V.  \int_{\RR} \frac{\theta_x(t_1,x_1) - \theta_x(t_1,y)}{|x_1-y|^{1+\gamma}} \dd y, \\
&= c P.V.  \int_{0}^\infty (\theta_x(t_1,x_1) - \theta_x(t_1,y) ) \left(\frac{1}{|x_1-y|^{1+\gamma}}  - \frac{1}{|x_1+y|^{1+\gamma}} \right) \dd y.
\end{align*}
By assumption we have $\theta_x(t_1,x_1)- \theta_x(t_1,y) \geq 0$ for all $y\geq 0$, and clearly  $|x_1+y| > |x_1-y|$ for all $x_1 >0$ and $y\geq 0$. Thus, $\Lambda^\gamma (\theta_x) (t_1,x_1) > 0$.
To summarize, we have that
\begin{align*}
\partial_t(\theta_x)(t_1,x_1) &\geq  A \eps e^{-At_1},\\
H\theta \, \partial_x(\theta_x)(t_1,x_1) &= 0, \\
\Lambda^\gamma (\theta_x) (t_1,x_1) &> 0.
\end{align*}
Recalling the choice $A  = 2 \sup_{t \leq t_0} \|\Lambda \theta(t,\cdot)\|_{L^\infty}$ and that $\theta_x(t_1,x_1) = \eps e^{-At_1}$, we obtain a contradiction with the equation \eqref{eq:theta:x:evolution} at the point $(t_1,x_1)$.
\end{proof}

\begin{theorem}
\label{thm:telescope:dissipative}
Let $\theta_0$ be even, non-negative, monotone decreasing on $(0,\infty)$. Let $\gamma \in (0,1/2)$. Then the initial value problem for \eqref{eq:CCF:dissipative} cannot have a global in time $C^{1}_x$ smooth solution.
\end{theorem}

\begin{proof}[Proof of Theorem~\ref{thm:telescope:dissipative}]
We modify the proof of Theorem~\ref{thm:telescope} by taking the test function $\eta(x)$ defined as
\begin{align}
\eta(x) = 
\begin{cases}
x^{-1-\alpha}&, 0 < x < 1, \\
x^{-2-\alpha}&, x>1,
\end{cases}
\label{eq:new:eta:def}
\end{align}
where the parameter $\alpha \in (0,1)$ will be later chosen suitably, in terms of $\gamma \in (0,1/2)$.
Note that $\eta(x) = - \partial_x \phi(x)$ for all  $x\in (0,1) \cup (1,\infty)$, where 
\[ 
\phi(x) = 
\begin{cases}
\frac{1}{1+\alpha} + \frac{1}{\alpha} (x^{-\alpha}-1) &, 0 < x < 1, \\
\frac{1}{1+\alpha} x^{-1-\alpha}&, x>1.
\end{cases}
\] 
and $\phi(x) \geq 0$.

Assume $\theta$ remains smooth on $[0,T)$. The following computations are then justified for all $t<T$. As in \eqref{eq:F:def}, we use the Lyapunov functional
\begin{align}
F(t) = \int_0^\infty  \eta(x) (\theta(t,0) -  \theta(t,x) ) \dd x.
\label{eq:F:def:2}
\end{align}
Note that in view of the non-integrable singularity of $\eta$ near the origin, as opposed to \eqref{eq:F:apriori}, ere we do not a-priori know that $F$ is a bounded function.

Then, similarly to \eqref{eq:F:dot:1}, since $\theta$ is even we have that
\[
\frac{dF}{dt}(t) = \int_0^\infty \eta(x) H\theta(t,x) \, \partial_x \theta(t,x) \dd x + \int_0^\infty \eta(x) \left( \Lambda^\gamma \theta(t,x) - \Lambda^\gamma \theta(t,0) \right) \dd x.
\]
The nonlinear term is bounded from below as in \eqref{eq:dtF:lower} by
\begin{align}
\int_0^\infty \eta(x) H\theta(t,x) \, \partial_x \theta(t,x) \dd x 
&\geq \frac{\log 3}{4\pi} \sum_{k\in \ZZ} \eta(2^{k+1})  (\theta(t,2^{k}) - \theta(t,2^{k+1}) )^2 \geq \frac{\log 3}{4\pi c_\alpha} (F(t))^2
\label{eq:nonlinearity:lower:telescoping}
\end{align}
where in the last inequality we have used \eqref{eq:F:upper:ancient}--\eqref{eq:F:upper} and the fact that 
\[
c_\alpha = \sum_{k \in \ZZ} \frac{(\phi(2^k))^2}{\eta(2^{k+1})} = \sum_{k<0}  \left(\frac{1}{1+\alpha} + \frac{2^{-k\alpha}-1}{\alpha}\right)^2 2^{(k+1)(1+\alpha)} + \sum_{k\geq 0} \frac{2^{-(2+2\alpha)k}}{(1+\alpha)^2} 2^{(k+1)(2+\alpha)}<\infty
\]
for $\alpha \in (0,1)$. Here we have implicitly used Lemma~\ref{lem:fractional:good:properties}.

In order to treat the nonlocal term, we appeal to a trick already present in~\cite{LiRodrigo08,Kiselev10}, namely
that $\Lambda^\alpha (\Lambda^\gamma \theta)(t,0) = \Lambda^{\alpha + \gamma}\theta(t,0)$. In view of the definition of $\eta$ and the evenness of $\theta$ we may write
\begin{align}
&\int_0^\infty \eta(x) (\Lambda^\gamma \theta(t,x) - \Lambda^\gamma \theta(t,0)) \dd x  \notag\\
&\quad =- \frac 12 \int_{\RR} \frac{\Lambda^\gamma \theta(t,0) - \Lambda^\gamma \theta(t,x)}{|x|^{1+\gamma}} \dd x - \int_{1}^{\infty} \frac{x-1}{x^{2+\alpha}} (\Lambda^\gamma\theta(t,x) - \Lambda^\gamma(t,0)) \dd x\notag\\
&\quad = - c_{\alpha,\gamma} \int_{\RR} \frac{\theta(t,0)-\theta(t,x)}{|x|^{1+\gamma+\alpha}} \dd x - \int_{\RR} \left( \frac{x-1}{x^{2+\alpha}} {\bf 1}_{x\geq 1}\right) \Lambda^\gamma \theta(t,x) \dd x + \Lambda^\gamma\theta(t,0) \int_1^\infty \frac{x-1}{x^{2+\alpha}} \dd x.
\label{eq:dissipative:lower:telescope}
\end{align}
At this stage we notice that the last term on the right side of \eqref{eq:dissipative:lower:telescope} is positive: indeed, $\theta(t,\cdot)$ attains its maximum at $x=0$, and thus $\Lambda^\gamma \theta(t,0) >0$, whereas the integral term in positive. Thus this term may be dropped for lower bounds. Also, the second term in the right side of \eqref{eq:dissipative:lower:telescope} may be bounded as
\[
\left| \int_{\RR} \left( \frac{x-1}{x^{2+\alpha}} {\bf 1}_{x\geq 1}\right) \Lambda^\gamma \theta(t,x) \dd x \right| = \left| \int_{\RR} \Lambda^\gamma \left( \frac{x-1}{x^{2+\alpha}} {\bf 1}_{x\geq 1}\right) \theta(t,x) \dd x \right| \leq c_{\alpha,\gamma}' \|\theta(t,\cdot)\|_{L^\infty} \leq c_{\alpha,\gamma}' \|\theta_0\|_{L^\infty}
\]
where we have used the $L^\infty$ maximum principle (Lemma~\ref{lem:maxprinciple}) and the fact that 
$\Lambda^\gamma \left( \frac{x-1}{x^{2+\alpha}} {\bf 1}_{x\geq 1}\right) \in L^1(\RR)$
when $\gamma \in (0,1)$. 
Thus, combining the above estimate with  \eqref{eq:dissipative:lower:telescope} we obtain
\begin{align}
\int_0^\infty \eta(x) (\Lambda^\gamma \theta(t,x) - \Lambda^\gamma \theta(t,0)) \dd x 
&\geq - \frac{c_{\alpha,\gamma}}{2} \int_{0}^\infty \frac{\theta(t,0)-\theta(t,x)}{x^{1+\gamma+\alpha}} \dd x  - c_{\alpha,\gamma}' \|\theta_0\|_{L^\infty} \notag\\
&\geq - \frac{c_{\alpha,\gamma}}{2} \int_{0}^1 \frac{\theta(t,0)-\theta(t,x)}{x^{1+\gamma+\alpha}} \dd x  - c_{\alpha,\gamma}'' \|\theta_0\|_{L^\infty}
\label{eq:dissipative:lower:telescope:2}
\end{align}
for some suitable constant $c_{\alpha,\gamma}''$. To conclude, we proceed as in \eqref{eq:F:upper:ancient}. 
We have 
\[ 
x^{-(1+\alpha+\gamma)} = \frac{1}{\alpha+\gamma} \partial_x \left( 1 - x^{-(\alpha+\gamma)}\right) 
\]
and thus
\begin{align}
(\alpha+\gamma) \int_{0}^1 \frac{\theta(t,0)-\theta(t,x)}{x^{1+\gamma+\alpha}} \dd x 
&= \int_0^1 \frac{1- x^{\alpha+\gamma}}{x^{\alpha+\gamma}} \left( -\partial_x \theta(t,x)\right) \dd x\notag\\
&\leq \sum_{k< 0}  \int_{2^{k}}^{2^{k+1}} \frac{1}{x^{\alpha+\gamma}} (-\partial_x \theta(t,x))\dd x \notag\\
&\leq  \sum_{k< 0} \frac{2^{-k(\alpha+\gamma)}}{(\eta(2^{k+1}))^{1/2}}  (\theta(t,2^k) - \theta(t,2^{k+1}) (\eta(2^{k+1}))^{1/2} \notag\\
&\leq  \left( \sum_{k< 0} 2^{-2k(\alpha+\gamma)} 2^{(k+1)(1+\alpha)} \right)^{1/2} \frac{F(t)}{c_\alpha^{1/2}}
\label{eq:main:dissipative:telescoping}
\end{align}
where in the last inequality we have used the Cauchy-Schwartz inequality and estimate \eqref{eq:nonlinearity:lower:telescoping}.
We emphasize that only at this stage a condition on the relationship between $\alpha,\gamma \in (0,1)$ emerges: in order for the sum on the right side of \eqref{eq:main:dissipative:telescoping} to converge, we need to choose
\[
\gamma < \frac{1-\alpha}{2}.
\]
This is the only reason that restricts the range of $\gamma$ to $(0,1/2)$, since we must have $\alpha>0$, which is in turn required to apply the composition of fractional powers of the Laplacian argument in \eqref{eq:dissipative:lower:telescope}.

Summarizing \eqref{eq:F:def:2}, \eqref{eq:nonlinearity:lower:telescoping}, \eqref{eq:dissipative:lower:telescope:2}, and \eqref{eq:main:dissipative:telescoping} we arrive at
\[
\frac{dF}{dt}(t) \geq \frac{(F(t))^2}{C_{\alpha,\gamma}} - C_{\alpha,\gamma} F(t) - C_{\alpha,\gamma} \|\theta_0\|_{L^\infty} 
\]
for a sufficiently large constant $C$ that depends solely on $\alpha$ and $\gamma$ and not on the data.
The proof is completed by choosing $\theta_0$ of $L^\infty$ norm $1$, but with $F(0)$ is sufficiently large. An example of such function can be obtained by slightly smoothing a cusp at the origin, of height $1$.
\end{proof}

\section{Singularity formation via De Giorgi} 
\label{sec:DeGiorgi}

We start with the following identity of Virial type.

\begin{lemma} \label{l:globaldissipation}
Let $\theta$ solve \eqref{eq:CCF}. Then
\[ 
\partial_t \int_{\R} \theta \dd x = -\|\theta\|_{\dot{H}^{1/2}}^2 = -\int \!\!\!\int_{\RR^2} \frac{|\theta(t,x)-\theta(t,y)|^2}{|x-y|^2} \dd x \dd y.
\]
\end{lemma}

\begin{proof}
Integration by parts.
\end{proof}

The Virial type identity of Lemma \ref{l:globaldissipation} can be used to play the same role as an energy dissipation inequality in a De Giorgi iteration scheme. This is the idea of this section. Using De Griogi's technique we derive a decay for $\theta$ in $L^\infty$, from which we can deduce that the solution must develop a singularity in finite time. Otherwise, its oscillation should be constant.

\begin{lemma} \label{l:gloabl-decay-unitscale}
Let $\theta$ solve \eqref{eq:CCF}. There is a constant $\eps_0 >0$ so that if $\int_\R \theta(0,x)^+ \dx \leq \eps_0$ then $\theta(1,x) \leq 1$ for all $x \in \R$.
\end{lemma}

Before giving the proof Lemma~\ref{l:gloabl-decay-unitscale} we recall the following interpolation.
\begin{lemma} \label{l:interpolation}
Let $f : \R \to \R$. The following interpolation holds
\[ \|f\|_{L^2} \leq C \|f\|_{L^1}^{1/2} \|f\|_{\dot H^{1/2}}^{1/2}.\]
\end{lemma}

\begin{proof}
We use the Fourier transform
\begin{align*}
\|f\|_{L^2}^2 = \int |\hat f|^2 \dd \xi &\leq \int_{-r}^r |\hat f|^2 \dd \xi + \frac 1 r \int_{\R \setminus [-r,r]} |\xi| |\hat f|^2 \dd \xi \leq 2r \|f\|_{L^1}^2 + \frac 1r \|f\|_{\dot{H}^{1/2}}^2
\end{align*}

Picking $r = \|f\|_{H^{1/2}} / \|f\|_{L^1}$ we finish the proof with $C=3$.
\end{proof}

\begin{proof}[Proof of Lemma~\ref{l:gloabl-decay-unitscale}]
We define the following truncations (like in De Giorgi's proof)
\[ \theta_k(t,x) = \left( \theta(t,x) - (1-2^{-k}) \right)^+.\]

Naturally, for any value of $k$, the function $\theta_k$ satisfies the same transport equation as $\theta$.
\[ \partial_t \theta_k + H\theta \; \partial_x \theta_k = 0.\]

Integrating by parts as in Lemma \ref{l:globaldissipation}, we obtain
\[ 
\begin{aligned}
\partial_t \int_\R \theta_k &= - \int_\R \Lambda \theta \; \theta_k \dd x
= - \int_\R \Lambda \theta_k \; \theta_k \dd x + \int_\R \Lambda \left(\theta_k - \theta + (1-2^{-k})\right)  \; \theta_k \dd x,
\end{aligned}
 \]

Note that $\theta_k > 0$ only at those points where $\theta > (1-2^{-k})$. These are also the points where the non negative function $\left(\theta_k - \theta + (1-2^{-k})\right)$ is equal to zero. Therefore, at these points $\Lambda \left(\theta_k - \theta + (1-2^{-k})\right) \leq 0$ and the second term on the right hand side is negative. Therefore
\begin{align} 
\label{e:truncated-dissipation}  
\partial_t \int_\R \theta_k  \leq - \int_\R \Lambda \theta_k \; \theta_k \dd x = -\|\theta_k\|_{\dot{H}^{1/2}}^2.
\end{align}

For the rest of the proof, we will construct a sequence of times $0 = t_0 < t_1 \leq t_2 \leq t_2 \leq \dots < 1$ such that the quantity
\[ a_k := \int_\R \theta_k(t_k,x) \dx,\]
converges to zero as $k \to \infty$. That means that if $t_\infty = \lim t_k \leq 1$, then $\theta(t_\infty,x) \leq 1$ and the result follows from the maximum principle of Lemma \ref{lem:maxprinciple}. 

We will make the construction so that for all $k>0$, $t_k$ belongs to the interval $(t_{k-1} , 1-2^{-k})$.
Assume we have constructed $t_k$ up to some value $k$. From \eqref{e:truncated-dissipation} and the mean value theorem, we can find some $t_{k+1} < 1 - 2^{-k-1}$ so that
\[ \|\theta_k(t_{k+1},\cdot)\|_{\dot{H}^{1/2}}^2 \leq 2^{k+1} a_k.\]

Also from \eqref{e:truncated-dissipation}, we know that $\|\theta_k(t_{k+1},\cdot)\|_{L^1} \leq \|\theta_k(t_k,\cdot)\|_{L^1} = a_k$. Using the interpolation of Lemma \ref{l:interpolation}, 
\begin{align*} 
\|\theta_k(t_{k+1},\cdot)\|_{L^2} &\leq \|\theta_k(t_{k+1},\cdot)\|_{L^1}^{1/2} \|\theta_k(t_{k+1},\cdot)\|_{H^{1/2}}^{1/2}, \\
&\leq 2^{(k+1)/4} \ a_k^{3/4}
\end{align*}
Therefore,
\begin{align}
a_{k+1} = \int (\theta_k - 2^{-k-1})^+ \dd x 
&\leq \|\theta_k\|_{L^2} \left| \left\{ \theta_k > 2^{-k-1} \right\} \right|^{1/2} \notag \\
&\leq 2^{k+1} \|\theta_k\|_{L^2}^2 \notag\\
&\leq 2^{5(k+1)/4} \ a_k^{3/2}.
\label{eq:ak:ineq:DeGiorgi}
\end{align}
Thus, we obtained a recurrence relationship for $a_k$ which converges to zero provided that 
\[ 
a_0 = \int \theta_0(t_0,x) \dx \leq \eps_0
\] 
is small enough, which concludes the proof of the Lemma.
\end{proof}

\begin{theorem} \label{l:globaldecay}
Let $\theta$ solve \eqref{eq:CCF}. Assume $\theta_0 = \theta(0,\cdot)$ is a non negative function in $L^1(\R)$. Then
\[ \|\theta(T,\cdot)\|_{L^\infty} \leq C \left( \frac {\|\theta_0\|_{L^1}}  T \right)^{1/2}.\]
\end{theorem}

\begin{proof}
Since $\theta_0 \geq 0$, from Lemma \ref{lem:maxprinciple}, $\theta(t,x) \geq 0$ for all $t,x$. In particular $\int_\R \theta(t,x) \dd x \geq 0$ for all $t \geq 0$. Let $\eps_0$ be the absolute constant from Lemma \ref{l:gloabl-decay-unitscale}. Consider the rescaled function $\tilde \theta$, which also solves \eqref{eq:CCF},
defined by
\[ 
\tilde \theta(t,x) = ab^{-1} \theta(at,bx),
\]
with $a = T$ and $b = \left(\|\theta_0\|_{L^1} T / \eps_0 \right)^{1/2}$. We easily check that
\begin{align*}
\|\tilde \theta(0,\cdot) \|_{L^1} &= \eps_0, \\
\|\theta(T,\cdot)\|_{L^\infty} &=  \frac ba \|\tilde \theta(1,\cdot)\|_{L^\infty}
\end{align*}

From Lemma \ref{l:globaldissipation}, we conclude that for all $t>0$,
\[ 
\int_0^t \| \tilde\theta(s,\cdot)\|_{\dot{H}^{1/2}}^2 \dd s \leq \|\tilde \theta_0\|_{L^1}.
\]
Applying Lemma \ref{l:gloabl-decay-unitscale}, we get that $\|\tilde \theta(1,\cdot)\|_{L^\infty} \leq 1$. Therefore
\[ 
\| \theta(T,\cdot) \|_{L^\infty} \leq \left( \frac{ \|\theta_0\|_{L^1} }{\eps_0 T} \right)^{1/2}
\]
which concludes the proof.
\end{proof}

\begin{corollary}
For any initial data $\theta_0 \geq 0$ which is integrable, there is no classical global solution $\theta$ to \eqref{eq:CCF}.
\end{corollary}

\begin{proof}
A classical solution would make $\max_\R \theta(t,\cdot)$ constant and $\min_\R \theta(t,\cdot) = 0$ for all time $t>0$. This contradicts Theorem~\ref{l:globaldecay}.
\end{proof}

\begin{remark}
The proof can be extended to the family of equations
\[ 
\partial_t \theta + (H \Lambda^s \theta) \; \theta_x = 0,
\]
provided that $-1 < s < 1$. Indeed, when $s\in (-1,1)$ we have that $s+1 >0$ and similarly to \eqref{e:truncated-dissipation} we arrive at
\[
\partial_t \int \theta_k \dd x \leq - \|\theta_k\|_{\dot{H}^{(1+s)/2}}^2.
\]
The interpolation inequality
\[
\|f\|_{L^2} \leq C_s \|f\|_{L^1}^{(s+1)/(s+2)} \|f\|_{\dot{H}^{(1+s)/2}}^{1/(s+2)}
\]
which may be proven as Lemma~\ref{l:interpolation} is, then leads to 
\[
a_{k+1} \leq C_s 2^{(k+1)(2s+5)/(s+2)} a_k^{(2s+3)/(s+2)}
\]
by repeating the argument in~\eqref{eq:ak:ineq:DeGiorgi}. Since $(2s+3)/(s+2) > 1$ for all $s \in (-1,1)$ the above inequality  is super-linear in $a_k$, and hence  assuming that $a_0 \leq \eps_0$ is sufficiently small we obtain that $a_k \to 0$ as $k \to \infty$. To conclude the proof of blowup, we note that the equation is invariant under the rescaling $\tilde\theta(t,x) = ab^{-1-s}\theta(at,bx)$. As in the proof of Theorem~\ref{l:globaldecay}, setting $T=a$ and $b = (\|\theta_0\|_{L^1} T/\eps_0)^{1/(2+s)}$ we arrive at the bound
\[
\|\theta(T,\cdot)\|_{L^\infty} \leq \frac{b^{1+s}}{a} = \frac{ \|\theta_0\|^{1/(2+s)}}{\eps_0^{1/(2+s)} T^{(1+s)/(2+s)}}.
\]
This decay of the $L^\infty$ norm, valid for $s>-1$, then yields the desired contradiction.
\end{remark}

\section{Singularity formation via barriers}
\label{sec:Barrier}

In this sections we show that the equation \eqref{eq:CCF} cannot have a global in time, $C^1$, even solution which is monotone on $(0,\infty)$. We do it using barriers, which is arguably the simplest way to prove that a singularity emerges in finite time for this equation.

\begin{lemma} \label{l:hilbert-monotonicity}
Let $f$ and $g$ be two $C^1$, even functions. Assume that for some point $x_0 > 0$ we have
\begin{itemize}
\item $f(x) \geq g(x)$ if $0 \leq x \leq x_0$.
\item $f(x) \leq g(x)$ if $x \geq x_0$.
\end{itemize}
Then $Hf(x_0) \leq Hg(x_0)$.
\end{lemma}

\begin{proof}
It follows from a direct application of formula \eqref{eq:Hilbert:even}. In this case, the integrand in the formula of $H(f-g)(x_0)$ is non positive at every point.
\end{proof}

\begin{lemma} \label{l:hilbert-for-tangent}
Let $\theta$ be $C^1$, even and monotone decreasing in $(0,\infty)$. Assume that the maximum
\[ \max \left\{ \theta(x) - A(1-|x|^{1/2})^+ : x \in \RR \right\},\]
is achieved at the point $x_0 \in (0,1/2]$.
Then 
\[ H \theta(x_0) \leq -c_0 A x_0^{1/2}\] for some positive universal constant $c_0$.
\end{lemma}

\begin{proof}
Let 
\[ h = \max \left\{ \theta(x) - A(1-|x|^{1/2})^+ \right\}.\]
We know $\theta(x_0) = A(1-|x_0|^{1/2})^+ + h$ and also $\theta(y) \leq A(1-|y|^{1/2})^+ + h$ for any $y \in (0,\infty)$.

Let $g : \R \to \R$ be the even function 
\[ 
g(x) = \begin{cases}
h + A(1-|x_0|^{1/2})^+ , & \text{if } |x| < x_0, \\
h + A(1-|x|^{1/2})^+ , & \text{if } |x| \geq x_0.
\end{cases}
\]
Since we assumed that $\theta$ is monotone decreasing on $[0,\infty)$ we have that $\theta \geq g$ on $[0,x_0]$, and by assumption $\theta \leq g$ on $[x_0,\infty)$. We apply Lemma \ref{l:hilbert-monotonicity} to get that $H\theta(x_0) \leq H g(x_0)$. We will conclude the proof by showing that 
\[
Hg(x_0) \leq -c_0 A |x|^{1/2}
\]
 for some positive constant $c_0$. This is an explicit computation. Using \eqref{eq:Hilbert:even}, we have
\[ Hg(x_0) = -A \frac {2 x_0}{\pi} \int_{x_0}^1 \frac{y^{1/2} - x_0^{1/2}}{y^2-x_0^2} \dd y - A\frac {2x_0}{\pi} (1-x_0^{1/2}) \int_1^\infty \frac 1 {y^2-x_0^2} \dd y.\]
Note that both terms are negative. Since $x_0 \in (0,1/2]$, the interval $(x_0,2x_0)$ is part of the first domain of integration. Then
\[ Hg(x_0) \leq -A \frac {2x_0} \pi \int_{x_0}^{2x_0} \frac{y^{1/2} - x_0^{1/2}}{y^2-x_0^2} \dd y = -c_0 A x_0^{1/2}
\]
where
\[
c_0 = \frac{2}{\pi} \int_1^2 \frac{z^{1/2}-1}{z^2-1} \dd z
\]
is a positive constant.
\end{proof}

\begin{lemma}
\label{l:barrier}
Let $\theta$ be a solution to \eqref{eq:CCF} which is $C^1$,  even and monotone decreasing for $x \in (0,\infty)$. Assume $h = \theta(0,1/2)$. Then
\begin{equation} 
\label{e:barrier}  
\theta(t,x) \leq h + \frac A t (1-|x|^{1/2})^+
\end{equation}
for all $x \in \R$ and $t>0$. Here $A$ is a universal constant.
\end{lemma}

\begin{proof}
Since $\theta$ is even and monotone decreasing on $(0,\infty)$, we have that $H\theta(t,x)$ is negative for $x>0$ and positive for $x<0$. Moreover $H\theta \, \theta_x \geq 0$ in $\R$. In particular, from \eqref{eq:CCF}, $\theta_t \leq 0$ and the function $\theta$ is monotone decreasing in time.

From the monotonicity properties of $\theta$ we immediately conclude that $\theta(x,t) \leq h$ whenever $x \geq 1/2$ and $t \geq 0$. Thus, the conclusion of the lemma could only be invalidated for $x < 1/2$.

For $t$ sufficiently small the right hand side in \eqref{e:barrier} will be larger than $\|\theta_0\|_{L^\infty}$ and thus the strict inequality holds for $0\leq x<1/2$. If this lemma were false, there would be a first time $t_0>0$ for which there exists an $x_0 \in [0,1/2]$ so that equality holds in \eqref{e:barrier}. Let us assume this in order to get a contradiction. We have
\begin{align*}
\theta(t_0,x_0) = h + \frac A {t_0} (1-|x_0|^{1/2})^+
\end{align*}
with $x_0 \in [0,1/2]$, and 
\begin{align*}
\theta(t,x) \leq h + \frac A {t} (1-|x|^{1/2})^+
\end{align*}
for all $x \in \R$ and all $0 \leq t \leq t_0$.
The point $x_0$ cannot be equal to zero since the differentiable function $\theta$ cannot be tangent from below at the cusp at $x=0$ of $h + \frac A {t_0} (1-|x|^{1/2})^+$. Moreover, at this point $(t_0,x_0)$ we have the classical first order conditions:
\begin{align*}
\theta_x(t_0,x_0) &= \partial_x \left[h + \frac A {t_0} (1-|x|^{1/2})^+ \right]_{|x=x_0} = -\frac{A}{2 t_0} x_0^{-1/2},
\end{align*}
and
\begin{align*}
\theta_t(t_0,x_0) &\geq \partial_t \left[h + \frac A {t} (1-|x_0|^{1/2})^+ \right]_{|t=t_0} \geq - \frac A {t_0^2}.
\end{align*}
Lastly, we apply Lemma \ref{l:hilbert-for-tangent} and obtain
\[ H\theta(t_0,x_0) \leq -c_0 \frac{A} {t_0} x_0^{1/2}.\]

Combining the three relations above we obtain
\[ \theta_t(t_0,x_0) + H\theta(t_0,x_0) \, \theta_x(t_0,x_0) \geq - \frac A {t_0^2} + \frac{A c_0}{2} \frac {A} {t_0^2}.\]
Choosing $A = 4/c_0$, the right hand side of the above inequality is strictly positive and we arrive at a contradiction with \eqref{eq:CCF}.
\end{proof}

\begin{theorem}
\label{thm:barrier}
Let $\theta(0,x)$ be even, monotone decreasing on $(0,\infty)$ and such that $\theta(0,1/2) < \theta(0,0)$. Then the equation \eqref{eq:CCF} cannot have a global in time $C^1$ solution.
\end{theorem}

\begin{proof}
Using Lemma \ref{l:barrier} we would obtain that $\theta(t,0) < \theta(0,0)$ for $t$ sufficiently large. However, for any $C^1$ solution to \eqref{eq:CCF}, $\theta(0,t)$ should be constant.
\end{proof}

\begin{remark}
The generalization of the above procedure to the active scalar equation \eqref{eq:CCF:smooth:drift}, where the drift velocity is given by $u = H \Lambda^s \theta$, is straightforward. The main difference is that the barrier $g(x)$ has now to be taken as 
\[ 
g_s(t,x) = \begin{cases}
h + \frac{A}{t}(1-|x_0|^{(1+s)/2})^+ , & \text{if } |x| < x_0, \\
h + \frac{A}{t}(1-|x|^{(1+s)/2})^+ , & \text{if } |x| \geq x_0.
\end{cases}
\]
Indeed, when $s\in (-1,1)$, the formula \eqref{eq:Hilbert:even}  becomes
\[
H \Lambda^s \theta(x) = c_s (1-s) \int_0^\infty \left(\theta(y) - \theta(x) \right) \frac{(y+x)|y+x|^s - (y-x)|y-x|^s}{|y^2 - x^2|^s (y^2 - x^2)} \dd y
\]
for all $x\geq 0$ and even functions $\theta$. Here $1/4\leq c_s \leq 1/2$ is a constant. In particular, Lemma~\ref{l:hilbert-monotonicity} holds without change, since $z \mapsto z^{1+s}$ is an increasing function for our range of $s$. Moreover, as in Lemma~\ref{l:hilbert-for-tangent}, an explicit computation shows that if the maximum of $\theta(x) - A (1- |x|^{(1+s)/2})^+$ is attained at a point $x_0 \in (0,1/2]$, we have 
\[
H\Lambda^s \theta(x_0) \leq H\Lambda^s g_s(1,x_0) \leq -c_{1,s} A x_0^{(1-s)/2}
\] 
for some $c_{1,s}>0$ that depends only on $s$. Here it is again important that $1+s>0$. To conclude, we proceed as in the proof of Lemma~\ref{l:barrier}, and use that 
\[
(\partial_t g_s + H\Lambda^s g_s\; \partial_x g_s)(t_0,x_0) \geq \frac{A}{t_0^2} \left(-1 + \frac{A(1+s) c_{1,s}}{2} \right)  > 0 
\]
for any $x_0 \in (0,1/2]$, once $A$ is chosen sufficiently large, depending only on $s$.
\end{remark}

\section{A possible regularization effect}
\label{sec:Conjectures}
\label{sec:Stationary}

In the proof of Section \ref{sec:CCF:revisited} we show that $\Lambda \theta$ must become $+\infty$ in finite time at any point where a local maximum of $\theta$ is attained. Intuitively, the graph of $\theta$ creates a cusp pointing up at every one of its local maximum points. This can be visualized numerically. We have experimented with a simple code which is available on the website \url{http://math.uchicago.edu/~luis/pde/hilbert.html}.

The numerical method which was used is inspired by monotone finite difference schemes for the Hamilton-Jacobi equation following ideas from \cite{Oberman06}. in the numerics, we observe that for any initial data $\theta_0$, the local maximums of $\theta$ flow into cusps pointing up while the local minimums tend to open up and seem to become more regular. If the initial data is very rough, the equation seems to regularize the solution in some H\"older norm, possibly $C^{1/2}$. The mechanism of this regularization process is perhaps similar to the Hamilton-Jacobi equation $u_t + u_x^2 = 0$. It is well known that viscosity solutions to this equation become immediately Lipschitz and semiconvex for any initial data $u_0$.

\subsection{The H\"older \texorpdfstring{$1/2$}{1/2} conjecture}

Note that the function 
\[ \theta(t,x) = -|x|^{1/2} - Ct \]
is in fact an exact solution to the equation \eqref{eq:CCF} away from $x=0$ for some value of $C$. It is not a weak solution in the sense of Section~\ref{d:weak-solution} only because it is unbounded. However, what we want to stress is the singularity of order exactly $1/2$. What this example suggests is that a cusp of order $1/2$ does not deteriorate by the flow. Moreover, we conjecture that solutions of the equation \eqref{eq:CCF} have an a priori estimate in $C^{1/2}$ of the following form.

\begin{conjecture}
\label{conj:Inviscid}
Let $\theta$ be any bounded solution of \eqref{eq:CCF} in the time interval $[0,T]$. The following estimate holds
\[ \sup_{x,y \in \R} \frac{\theta(T,x) - \theta(T,y)}{|x-y|^{1/2}} \leq \frac{C \|\theta\|_{L^\infty}^{1/2} |x-y|^{1/2}}{T^{3/2} }.\]
Here $C$ is a universal constant.
\end{conjecture}

Conjecture \ref{conj:Inviscid} is supposed to hold for as long as the solution exist. Moreover, the vanishing viscosity limits of \eqref{eq:CCF}, with any bounded datum $\theta_0$, would satisfy Conjecture \ref{conj:Inviscid} for all $T>0$. We make this explicit as follows. For any $\eps>0$, we define $\theta^\eps$ to be the solution to
\begin{align*}
\theta^\eps_t + H\theta^\eps \partial_x \theta^\eps - \eps \partial_{xx} \theta^\eps &= 0, \\
\theta^\eps(0,x) &= \theta_0(x).
\end{align*}
Then, we believe $\theta^\eps$ satisfies the a priori estimate
\[ \sup_{x,y \in \R} \frac{\theta^\eps(t,x) - \theta^\eps(t,y)}{|x-y|^{1/2}} \leq \frac{C \|\theta_0\|_{L^\infty}^{1/2}  |x-y|^{1/2}}{t^{3/2}},
\]
where $C$ is a universal constant independent of $\eps$. The precise form of the right hand side can be easily derived from scaling considerations.

As $\eps \to 0$, we expect $\theta^\eps$ converges to a weak solution $\theta$ of the equation \eqref{eq:CCF} which satisfies the same estimate.

Assuming the result on the above conjecture is still valid after adding fractional dissipation $\Lambda^\gamma$ to the equations, in view of the natural parabolic scaling, the $C^{1/2}$ regularity threshold makes the power  $\gamma = 1/2$ critical. Applying Theorem 1.1 in \cite{Silvestre10b} 
it would follow that $\theta$ is $C^1$ and therefore a global classical solution.

The following conjecture, is in fact a consequence of Conjecture \ref{conj:Inviscid}.

\begin{conjecture}
\label{conj:Supercritical}
Consider $\theta_0$ as in Conjecture~\ref{conj:Inviscid}. The vanishing viscosity limits of the supercritically dissipative Hilbert model \eqref{eq:CCF:dissipative} with $\gamma >1/2$, lie in $L^\infty([0,\infty); C^\infty)$.
\end{conjecture}

\subsection{H\"older continuity for a stationary problem}

We have not been able to prove Conjecture \ref{conj:Inviscid}. We can, however, prove a stationary version of the same result under the monotone regime. In this case we consider a bounded right hand side.

\begin{theorem}
\label{thm:stationary}
Consider a $C^1$ smooth solution $\theta$ of the stationary problem
\begin{align}
H\theta \, \partial_x \theta = f
\label{eq:stationary:hilbert}
\end{align}
where $0 \leq f \in L^\infty$. If $\theta$ is even, monotone decreasing away from the origin, and non-negative, then the following estimate holds
\[ 
|\theta(x_1) - \theta(x_2)| \leq C \|f\|_{L^\infty}^{1/2} |x_1 - x_2|^{1/2}
\]
for all $x_1 \neq x_2 \in \RR$. Here $C>0$ is a universal constant. 
\end{theorem}
\begin{proof}
 We appeal to Lemma~\ref{lem:nonlinear:positivity}. Let $0<x_2<x_1$ be arbitrary.
 Then integrating \eqref{eq:stationary:hilbert} from $x_2$ to $x_1$ we arrive at
\begin{align}
(x_1-x_2) \|f\|_{L^\infty} \geq \int_{x_2}^{x_1} f \dd x 
&= \int_{x_2}^{x_1} H\theta \, \partial_x \theta \dd x\notag\\
&\geq \frac 14  (\theta(x_2)-\theta(x_1))^2 \log \left( \frac{x_1+x_2}{x_1-x_2} \right).
\label{eq:stationary:main:bound}
\end{align}

Now, for $0 \leq x_2 < x_1$, by the continuity of $\theta$ at $x_2$ and the estimate \eqref{eq:stationary:main:bound}, we have that
\begin{align}
\theta(x_2) - \theta(x_1) &= \sum_{k=0}^\infty \theta(x_2 + 2^{-k-1} (x_1-x_2) ) - \theta(x_2 + 2^{-k} (x_1-x_2)), \notag \\
&\leq 2 \|f\|_{L^\infty}^{1/2}  \sum_{k=0}^\infty \left( \frac{2^{-k-1} (x_1-x_2)}{\log\left( 3 + 2^{k+2} \frac{x_2}{x_1-x_2}\right)} \right)^{1/2}, 
\notag \\
&\leq \frac{2 (1+\sqrt 2)}{ \log 3} \|f\|_{L^\infty}^{1/2}(x_1-x_2)^{1/2}.
\label{eq:stationary:main:bound:2}
\end{align} 
We have proven  the inequality whenever $x_1$ and $x_2$ are both positive. The general case follows easily given that $\theta$ is an even function.
\end{proof}

\begin{remark}
 The proof of Theorem~\ref{eq:stationary:hilbert} applies to even, monotone decreasing away from the origin, continuous at the origin, stationary solutions 
 of 
 \[
 H\Lambda^s \theta \partial_x \theta  = f
 \]
 when $0 \leq f \in L^\infty$ and $s \in (-1,1)$.
 In this case it follows from Lemma~\ref{lem:nonlinear:positivity:drift} and the argument in \eqref{eq:stationary:main:bound:2} that
\[
|\theta(x_1) - \theta(x_2)| \leq \frac{8 \|f\|_{L^\infty}^{1/2} }{\left( \frac{1-s}{s} (1 - 3^{-s} )  \right)^{1/2}} |x_1 - x_2|^{(1+s)/2}
\]
for all $0 \leq x_2 < x_1$.
\end{remark}

\subsection{A related conjecture: one sided bounds for \texorpdfstring{$\Lambda \theta$}{\Lambda \theta}}

To conclude the section we discuss what we believe is an intimate connection between Conjecture~\ref{conj:Inviscid} and lower bounds for $\Lambda \theta(t,\cdot)$. We begin with another conjecture about the evolution~\eqref{eq:CCF}, which has the geometric meaning that the possible cusps forming in finite time will always open downwards. The formation of only downward opening cusps is consistent with the numerical simulations of \eqref{eq:CCF}.

\begin{conjecture}
\label{conj:Lambda:lower:bound}
Consider $\theta_0$ as in Conjecture~\ref{conj:Inviscid}. The vanishing viscosity limits of \eqref{eq:CCF}  obey
\begin{align} \label{eq:lower:conj}
\Lambda \theta(t,x) \geq - A(t)
\end{align}
for some $A(t) = A(t,\theta_0)\geq 0$,
for every $t>0$.
\end{conjecture}

On the one hand, as discussed in Remark~\ref{rem:Lambda:lower:bound} of Section~\ref{sec:CCF:revisited}, the integral identity of Proposition~\ref{p:identity} is not sufficient (just barely) for showing that a lower bound for $\Lambda \theta$ is propagated forward in time by the evolution~\ref{eq:CCF}. 

On the other hand, if  a lower bound of the type \eqref{eq:lower:conj} would hold, we show here that in the presence of additional symmetries (evenness and monotonicity away from the origin), the H\"older-$1/2$ continuity of $\theta$ directly follows (as claimed by Conjecture~\ref{conj:Inviscid}), but only at points away from the origin. 
 
\begin{theorem}
\label{thm:Holder}
Let $\theta : \R \to \R$ be even, continuous, non-negative, and monotone decreasing in $[0,+\infty)$. Assume $\Lambda \theta \geq -A$. Then, we have
\begin{equation} \label{e:c12estimate}  \theta(x_1) - \theta(x_2) \leq C \max\left(\|\theta\|_{L^\infty}, A x_1 \right) \left(\frac{x_1-x_2}{x_1} \right)^{1/2}
\end{equation}
for all  $0 < x_1 < x_2$, where $C>0$ is a universal constant.
\end{theorem}

We note that the above result is a property of functions, and the evolution \eqref{eq:CCF} is not used here.
Before giving the proof of Theorem~\ref{thm:Holder} we discuss two auxiliary lemmas.

\begin{lemma}
\label{lem:Lambda:r}
The function $r(x) = x_+^{1/2}$ satisfies
\[ \Lambda r(x) = \begin{cases}
-\frac 12 |x|^{-1/2} & \text{if } x<0,\\
0  & \text{if } x>0.
\end{cases}\]
\end{lemma}

\begin{proof}
This is a classical computation. The function 
\[ 
u(x,y) = \left( \frac{\sqrt{x^2+y^2} + x} 2 \right)^{1/2}
\]
obeys $\lap_{x.y} u = 0$ for $y>0$, i.e., it is harmonic in the upper half plane. Moreover,
\begin{align*}
u(x,0) = x_+^{1/2}, \qquad 
u_y(x,0) =\begin{cases}
\frac 12 |x|^{-1/2} & \text{if } x<0,\\
0  & \text{if } x>0.
\end{cases}
\end{align*}
Thus, we recover $\Lambda r$ as the Dirichlet to Neumann map corresponding to $u$.
\end{proof}

\begin{lemma} \label{lem:barrier}
The function
\[ b(x) = \min \left(1, (|x|-1)_+^{1/2}\right)\]
satisfies
\begin{align}
\Lambda b(x) \geq \frac{1}{2 \pi} \qquad \mbox{for all} \qquad 1 < x \leq 2.
\label{eq:b:lower:bound}
\end{align}
\end{lemma}
\begin{proof}
The proof of this lemma follows from the explicit commutation in Lemma~\ref{lem:Lambda:r} and a comparison principle. First we note that $b(x) = r(x-1)$ for all $x \in [-1,2]$. Then, for all $x\in (1,2]$ we obtain
\begin{align}
\Lambda b(x ) = \Lambda b(x) - \Lambda r(x-1) 
&= \frac{1}{\pi} p.v. \int_{\RR} \frac{b(x) - b(y) - r(x-1) + r(y-1)}{(x-y)^2} \dd y \notag\\
&=  \frac{1}{\pi} p.v. \int_{\RR} \frac{r(y-1) - b(y)}{(x-y)^2} \dd y \notag\\
&= \frac{-1}{\pi} \int_{-\infty}^{-1} \frac{b(y)}{(x-y)^2} \dd y + \frac{1}{\pi} \int_{2}^\infty \frac{(y-1)^{1/2} - 1}{(x-y)^2} \dd y\notag\\
&\geq \frac{-1}{\pi} \int_{1}^{\infty} \frac{1}{(x+y)^2} \dd y + \frac{1}{\pi} \int_{2}^\infty \frac{(y-1)^{1/2} - 1}{(x-y)^2} \dd y \notag\\
&\geq \frac{-1}{\pi} \int_{1}^{\infty} \frac{1}{(1+y)^2} \dd y + \frac{1}{\pi} \int_{2}^\infty \frac{(y-1)^{1/2} - 1}{(1-y)^2} \dd y  \notag\\
&= - \frac{1}{2\pi} + \frac{1}{\pi} = \frac{1}{2 \pi}.
\end{align}
This completes the proof of the lemma.
\end{proof}

\begin{proof}[Proof of Theorem~\ref{thm:Holder}]
Without loss of generality, we prove estimate \eqref{e:c12estimate} for the case $x_1 = 1$. The general case follows by scaling.

Define the function \[
f(x) = \theta(x) + \max \left(\|\theta\|_{L^\infty} , 2\pi A  \right) b\left( x \right).
\]
We prove that the global minimum value of $f$ is achieved at $x=1$. From that, the lemma will follow.

For $|x| > 2$, we have $f(x) \geq \|\theta\|_{L^\infty} \geq f(1)$. Since $f$ is continuous, it must thus achieve its global minimum in the interval $[-2,2]$.

Since the second term in the definition of $f$ is constant for $x \in (-1,1)$ and the first term is monotone, we have that $f(x) \geq f(1)$ for all $x \in [-1,1]$. Therefore, the global minimum of $f$ must be achieved on the interval $[1,2]$. Here we also used that $f$ is even.

Assume that $f$ achieves its global minimum at some point $x_0 \in (1,2]$. Then by the definition of $A$ we have that
\[ 
0 > \Lambda f(x_0) = \Lambda \theta(x_0) + 2\pi A  \Lambda b\left( x_0 \right) 
\geq A \left( -1  +  2\pi \Lambda b\left( x_0 \right) \right)
\geq 0,
\]
where in the last inequality we have used \eqref{eq:b:lower:bound}. This yields a contradiction which means that the minimum must be achieved at $x_0= 1$. Therefore, for any $x > 1$, we have $f(1) \leq f(x)$, which by definition means that
\[ 
\theta(1) - \theta(x) \leq C \max\left(\|\theta\|_{L^\infty},A \right) \left( x - 1\right)^{1/2}.
\]
As noted earlier, the proof of the lemma now follows from the above estimate and re-scaling.
\end{proof}

\begin{remark}
We conclude the subsection by pointing out that in the presence of symmetries, additional information may be obtained for the endpoint Sobolev embeddings. 
For instance, recall that in $1D$ the Sobolev embedding of $H^{1/2}$ in $L^\infty$ barely fails.
However, under the additional assumption of monotonicity away from the origin, the boundedness of the $H^{1/2}$ norm implies the continuity of the solution. Indeed, letting $0<x_2<x_1$ we have
\begin{align}
c \|\theta\|_{\dot{H}^{1/2}}^2  
= \int \!\!\! \int \frac{(\theta(y)-\theta(z))^2}{(y-z)^2} \dd y \dd z 
&\geq \int_{0 < y < x_2} \int_{z >x_1} \frac{(\theta(y)-\theta(z))^2}{(y-z)^2} \dd y \dd z \notag\\
&\geq (\theta(x_1)-\theta(x_2))^2 \int_{0 < y < x_2} \int_{z>x_1} \frac{1}{(y-z)^2} \dd y \dd z \notag\\
&=  (\theta(x_1)-\theta(x_2))^2  \log \frac{x_1}{x_1-x_2}
\label{eq:C12:22}
\end{align}
for some universal constant $c>0$. That is, when $\theta$ is even, positive, and decreasing away from the origin, we have
 \begin{align*}
0 \leq \theta(x_2) - \theta(x_1) \leq c \|\theta\|_{\dot{H}^{1/2}} \left( \log \frac{x_1}{x_1-x_2} \right)^{-1/2}
\end{align*}
for any $0<x_2<x_1$, where $C>0$ is a universal constant. The above estimate is in direct analogy with the fact that in 2D, $H^1$ functions that obey the maximum principle on every ball have a logarithmic modulus of continuity~\cite{SereginSilvestreSverakZlatos12}.
\end{remark}

\subsection{Weak solutions}
\begin{definition} \label{d:weak-solution}
Let $\theta_0 \in L^1 \cap L^\infty$ be non-negative, even, decreasing on $(0,\infty)$, and let $T>0$. We call a non-negative, even, decreasing on $(0,\infty)$ function 
\[
\theta \in L^\infty(0,T; L^1(\RR) \cap L^\infty(\RR))  \cap L^2(0,T;\dot{H}^{1/2}(\RR))
\]
a weak solution to the initial value problem \eqref{eq:CCF} with initial datum $\theta_0$ on $[0,T)$, if
\[
\int_0^T \int_{\RR} \left( \theta  \partial_t \phi - H \theta\, \theta \partial_x \phi + \frac 12 \theta^2 \Lambda \phi + \frac 12 D[\theta] \phi \right) \dd x \dd t + \int_\RR \theta_0 \phi \dd x= 0
\]
for any $\phi \in C_0^\infty([0,\infty) \times \RR)$, where
\begin{align}
D[\theta](x) = C \int_{\RR} \frac{(\theta(x) - \theta(y))^2}{(x-y)^2} \dd y
\end{align}
where $C$ is a universal constant.
\end{definition}
The above definition of a weak solution is natural in view of the pointwise identity
\[
H \theta \, \partial_x \theta 
= \partial_x (H \theta \, \theta) - \partial_x H \theta \, \theta
= \partial_x (H \theta \, \theta) + \Lambda \theta \, \theta 
= \partial_x (H \theta \, \theta) + \frac 12 \Lambda (\theta^2)  + \frac 12 D[\theta].
\]
Note that in one dimension the space $H^{1/2} \cap L^\infty$ is an algebra.
Also, since
\[
D[\theta] \geq 0, \qquad \mbox{and} \qquad \int_{\RR} D[\theta](x) \dd x = \| \theta\|_{\dot{H}^{1/2}}^2
\]
we have
\[
\left| \int_\RR D[\theta] \phi \dd x \right| =  \int_\RR D[\theta] |\phi| \dd x \leq \|\phi\|_{L^\infty} \|\theta\|_{\dot{H}^{1/2}}^2
\]
so that all terms in the distributional definition of the weak solution are well-defined. 

\begin{remark}
Consider a sequence of viscosity approximations, i.e. global in time smooth (in positive time) solutions $\theta^\eps$ of 
\[
\partial_t \theta^\eps + H\theta^\eps\, \partial_x \theta^\eps = \eps \Delta \theta^\eps, \qquad \theta^\eps = \theta_0,
\]
for $\eps \in (0,1]$. Using the $L^\infty$ (cf.~Lemma~\ref{lem:maxprinciple}) and the $L^1$ energy inequality (cf.~Lemma~\ref{l:globaldissipation}) is not difficult to see that
\[
\{ \theta^\eps \}_{\eps >0} \mbox{ is uniformly bounded in }L^\infty_t L^1_x \cap L^\infty_t L^\infty_x  \cap L^2_t \dot{H}^{1/2}_x
\]
globally in time. Moreover, we have the energy inequality
\[
\|\theta^\eps(T)\|_{L^1} + \int_0^T \|\theta^\eps(t)\|_{\dot{H}^{1/2}}^2 \dd t \leq \|\theta_0\|_{L^1}
\]
for any $T>0$.
It seems however that one is missing a bit more regularity, in order to ensure that the sequence $\theta^\eps$ converges along a subsequence to a weak solution $\theta$ of \eqref{eq:CCF}. The main missing part is the convergence of $D[\theta^\eps] \to D[\theta]$ in $L^1_{loc,t,x}$.
\end{remark}

\begin{remark}
One-sided conditions on the derivative of the drift velocity are known~\cite{Evans98} to ensure the uniqueness of {\em weak solutions} to the Burgers and Hamilton-Jacobi equations (the extremal equation on the scale \eqref{eq:CCF:smooth:drift}). 
For Burgers solutions the shocks occur in one direction, which follows since the derivative of the solution obeys a one-side bound. For the Hamilton-Jacobi ($s=1$) equation we have one-sided bounds for the second derivatives. One may expect that a one sided bound on $\Lambda \theta = - \partial_x (H\theta)$ is thus relevant in establishing the uniqueness of weak solutions. Insofar this remains open. 
\end{remark}

\begin{remark}
In \cite{CastroCordoba09} the authors construct the {\em expanding semicircle solution}
\begin{align*}
\theta&: (0,+\infty) \times \R \to \R, \\
\theta&(t,x) = - C (1 - x^2/t^2)_+^{1/2}
\end{align*}
where $C>0$ is a universal constant. This is a weak solution to \eqref{eq:CCF} that converges to zero as $t \to 0^+$ for all $x \neq 0$. Using the transformation $t \mapsto -t$ and $\theta \mapsto -\theta$
we also derive the {\em shrinking semicircle solution}
\begin{align*}
\theta&: (-\infty,0) \times \R \to \R, \\
\theta&(t,x) = C (1 - x^2/t^2)_+^{1/2}
\end{align*}
which is also a weak solution and converges to zero as $t \to 0^-$ for $x \neq 0$.

Based on our numerical computations, only the former appears to be stable. Additionally for the later solution the lower bound on $\Lambda \theta$ does not hold.
This suggests that the shrinking semicircle may not be a vanishing viscosity limit.

We can compare this situation with Burgers or Hamilton-Jacobi equation in the sense that time reversibility is broken in the vanishing viscosity limit by the entropy condition or the viscosity solution condition respectively.
\end{remark}

\section*{Acknowledgements}
The authors are thankful to Peter Constantin, Diego C\'ordoba, Hongjie Dong, Alexander Kiselev, and Alexis Vasseur for stimulating discussions about the model \eqref{eq:CCF}. VV is grateful to the hospitality of the Department of Mathematics at the University of Chicago where part of this work was completed. The work of LS was in part supported by NSF grants DMS-1254332 and DMS-1065979. The work of VV was in part supported by NSF grant DMS-1348193.


\end{document}